\documentclass[twoside,11pt]{article}

\usepackage{blindtext}

\usepackage{float}
\usepackage[abbrvbib,preprint]{jmlr2e}

\usepackage{enumitem} 

\newtheorem{assumption}{Assumption}

\newcounter{biscompt}

\newtheorem{bis}[biscompt]{Assumption}

\def\R{{\mathbb{R}}}

\def\N{{\mathbb{N}}}

\def\1{\mathbb{1}}

\def\xcont{(x_t)_{t\ge 0}}
\def\mart{(M_t)_{t\ge 0}}
\def\xz{(x_t,z_t)_{t\ge 0}}
\def\xdisc{\{x_n \}_{n\ge 0}}
\def\xdisctilde{\{\tilde{x}_k \}_{k\in \N}}
\def\minset{\mathcal{X}^\ast}
\def\tk{\{ T_k \}_{k\in \N^\ast}}

\def\AA{\mathbb{A}}
\renewcommand{\P}{\mathbb{P}}
\newcommand{\bpar}[1]{\left(#1\right)}

\newcommand{\E}[1]{\mathbb{E}\left[#1 \right]}

\newcommand{\norm}[1]{\left\lVert#1\right\rVert}
\newcommand{\dotprod}[1]{\left< #1\right>}

\newcommand{\off}[1]{}

\newcommand{\bigO}{\mathcal{O}}

\usepackage{hyperref}
\usepackage{amsmath}

\usepackage{xcolor}
\usepackage{lastpage}

\ShortHeadings{Continuized Nesterov acceleration}{Hermant, Aujol, Dossal, Huang, Rondepierre}
\firstpageno{1}

\begin{document}

\title{Continuized Nesterov Acceleration for Non-Convex Optimization}

\author{\name Julien Hermant \email julien.hermant@math.u-bordeaux.fr\\
       \addr Institute of mathematics \\
       University of Bordeaux\\
       Talence, France
       \AND
       \name Jean-François Aujol \email Jean-Francois.Aujol@math.u-bordeaux.fr \\
       \addr Institute of mathematics \\
       University of Bordeaux\\
       Talence, France
       \AND
         \name Charles Dossal \email dossal@insa-toulouse.fr \\
       \addr Department of Mathematical Engineering and Modeling \\
       INSA Toulouse\\
       Toulouse, France
       \AND
        \name Lorick Huang \email lhuang@insa-toulouse.fr \\
       \addr Department of Mathematical Engineering and Modeling \\
       INSA Toulouse\\
       Toulouse, France
       \AND
        \name Aude Rondepierre \email aude.rondepierre@insa-toulouse.fr \\
       \addr Department of Mathematical Engineering and Modeling \\
       INSA Toulouse\\
       Toulouse, France}
\editor{My editor}

\maketitle

\begin{abstract}
In convex optimization, continuous-time counterparts have been a fruitful tool for analyzing momentum algorithms. Fewer such examples are available when the function to minimize is non-convex.
In several cases, discrepancies arise between the existing discrete-time results — namely those obtained for momentum algorithms — and their continuous-time counterparts, with the latter typically yielding stronger guarantees.
We argue that the continuized framework \citep{even2021continuized}, mixing continuous and discrete components, can tighten the gap between known continuous and discrete results. This framework relies on computations akin to standard Lyapunov analyses, from which are deduced convergence bounds for an algorithm that can be written as a Nesterov momentum algorithm with stochastic parameters. 
In this work, we extend the range of applicability of the continuized framework, \textit{e.g.} by allowing it to handle non-smooth Lyapunov functions. We then strengthen its trajectory-wise guarantees for linear convergence rate, deriving finite time bounds with high probability and asymptotic almost sure bounds.
We apply this framework to the non-convex class of strongly quasar-convex functions. Adapting continuous-time results that have weaker discrete equivalents to the continuized method, we improve by a constant factor the known convergence rate, and relax the existing assumptions on the set of minimizers.
\end{abstract}

\begin{keywords}
  Non-Nonvex Optimization, Accelerated First-Order Algorithms, Stochastic Optimization, Continuized Algorithms, Jump Process
\end{keywords}

\section{Introduction}
We consider the following unconstrained minimization problem
\begin{equation}\label{prob}\tag{P}
    \min_{x \in \mathbb{R}^d} f(x):= f^\ast
\end{equation}
where $f: \mathbb{R}^d \to \mathbb{R}$ is a differentiable function such that $\arg \min_{x \in \mathbb{R}^d} f(x)$ is non-empty. As there is often no closed-form solution to Problem \eqref{prob}, optimization algorithms are used to approximate its solution.
When it comes to high-dimensional problems, \textit{i.e.} $d>>1$, the so-called \textit{first-order algorithms} that make use of the gradient information, such as \textit{gradient descent}, are popular.
It can be explained by their dimension-free complexity bounds, and the relative cheapness of their iterations, \textit{e.g.} compared to second order methods which involve the computation of the Hessian matrix.
To solve Problem \eqref{prob} efficiently, we seek algorithms that converge as fast as possible. If gradient descent benefits from a simple formulation, it however typically exhibits relatively slow convergence, such that several mechanisms exist in order to accelerate its convergence \citep{hinton2012rmsprop,kingma2014adam}.
 A widely used acceleration technique consists in adding a momentum mechanism to
gradient descent. The two main variants are Polyak's Heavy Ball \citep{POLYAK19641} and Nesterov momentum \citep{nesterov1983method}.
In convex optimization, momentum algorithms are known to attain the optimal worst-case complexity rate among first-order methods in many settings
\citep{POLYAK19641,nesterov1983method,nemirovskij1983problem,taylor2023optimal,aujdossrondPL,adr_qsc}. However, many problems \eqref{prob} of practical interest are such that $f$ is non-convex \citep{li2018visualizing,dauphin2014identifying,ge2016matrix,ge2017no,bhojanapalli2016global}. While it is empirically observed that the acceleration property of momentum still holds despite non-convexity in various settings \citep{sutskever2013importance,he2016deep}, there exist some classes of non-convex functions such that gradient descent already achieves optimal convergence bounds among first-order algorithms \citep{lowerboundI,PLlowerbound}.
Showing the benefit of momentum algorithms in a non-convex regime is thus an active field of research \citep{hadjisavvas2025heavy,wang2023continuizedaccelerationquasarconvex,li2023restarted,jin2018accelerated,hinder2020near,hermant2024study,guptanesterov,fu2023and,okamura2024primitive,lara2022strongly}.

\subsection{ODEs Momentum and their Limit for Algorithm Design}\label{sec:classic_ode_framework}
Different canonical formulations of momentum algorithms coexist in the literature, see a brief discussion in Appendix~\ref{app:momentum_form}. In this work, we study the following \citep{nesterovbook}
\begin{align}\label{alg:nest_classic}\left\{
    \begin{array}{ll}
       \tilde y_k &= \alpha_k \tilde x_k + (1- \alpha_k)\tilde z_k \\
        \tilde x_{k+1} &= \tilde y_k - s\nabla f(\tilde y_k) \\
        \tilde z_{k+1} &= \beta_k\tilde z_k + (1 - \beta_k)\tilde y_k - \eta_k \nabla f(\tilde y_k)
    \end{array}
\right.
\end{align}
 Note that
$\xdisctilde$ reduces to gradient descent for the choice $\alpha_k = 1$, $\forall k\in \N$. 
As previously mentioned, it is known to accelerate over gradient descent in many settings in convex optimization. In particular, for $\mu$-strongly convex functions with $L$-Lipschitz gradient, $\xdisctilde$ generated by \eqref{alg:nest_classic} properly parameterized ensures $f(\tilde x_k)-f^\ast = \bigO((1-\sqrt{\mu/L})^k)$ \citep{nesterovbook}. This should be compared with the bound $\bigO((1-\mu/L)^k)$ using gradient descent, which can be significantly worse as typically $\mu \ll  L$ (this is particularly true in the high dimensional setting). It is thus common to mention this gain of a squared dependence on $\mu/L$ as \textit{acceleration}.

A common approach to studying optimization algorithms is to consider their continuous-time counterparts. By letting the stepsize $s$ in \eqref{alg:nest_classic} tend to zero,
we obtain a continuous-time limit, closely related to the Heavy Ball equation (see Appendix~\ref{app:ode_link}), that is
\begin{align}\label{eq:nest_2_var}\left\{
    \begin{array}{ll}
        \dot{x}_t &= \eta_t(z_t-x_t) - \gamma_t \nabla f(x_t) \\
        \dot{z}_t &= \eta'_t(x_t-z_t) - \gamma'_t \nabla f(x_t) 
    \end{array}
\right.
\end{align}
with $t \in \R_+$ for some choices of parameters $\eta_t,\eta'_t,\gamma_t$ and $\gamma'_t$, see the derivation in \citep[Appendix A.1]{pmlr-v202-kim23y}.
A strategy to obtain convergence results is then to find a suitable function $\varphi : \mathbb{R}^d \to \R$, typically depending on the function we want to minimize, such that the Lyapunov function $E_t := \varphi(x_t,z_t,t)$ will decrease with $t$. 
Working in continuous time is often more convenient, since one can analyze the derivative $\dot{E}_t$.
To deduce a result for \eqref{alg:nest_classic}, one can try to mimic the analysis of the continuous counterpart by constructing a discrete Lyapunov function $\{ E_k \}_{k \in \N}$ inspired by $(E_t)_{t \ge 0}$. As an illustration, considering $\mu$-strongly convex functions, an appropriate choice $(E_t)_{t \ge 0}$ applied to the solution $\xcont$ of \eqref{eq:nest_2_var} (properly parameterized), yields the rate $f(x_t)-f^\ast = \bigO(e^{-\sqrt{\mu t}})$ \citep{siegel2021accelerated}. A straightforward discrete analogue of this Lyapunov function leads to the bound $f(\tilde x_k)-f^\ast = \bigO((1-\sqrt{\mu/L})^k)$, for $\xdisctilde$ generated by \eqref{alg:nest_classic} properly tuned,  see Table~\ref{tab:illustration}. 
\begin{table}[]
\begin{tabular}{l|l|l}
           & Lyapunov Function                                                         & Convergence Bound                                  \\ \hline
Continuous & $E_t = f(x_t)-\min f(x) + \frac{\mu}{2}\norm{z_t-x^\ast}^2$               & $f(x_t)-f^\ast = \bigO(e^{-\sqrt{\mu t}})$                \\ \hline
Discrete   & $E_k = f(\tilde x_k)-\min f(x) + \frac{\mu}{2}\norm{\tilde z_k-x^\ast}^2$ & $f(\tilde x_k)-f^\ast = \bigO((1-\sqrt{\mu/L})^k)$
\end{tabular}
\caption{Example of "convergence transfer" between continuous and discrete settings. $\xcont$, respectively $\xdisctilde$, is generated by \eqref{eq:nest_2_var}, respectively \eqref{alg:nest_classic}, properly tuned. $f$ is supposed to be $\mu$-strongly convex in both cases, and is also $L$-smooth in the discrete case. }\label{tab:illustration}
\end{table}
Both rates correspond to linear convergence, with rates $\sqrt{\mu}$ and $\sqrt{\mu/L}$, respectively. 
In this sense, the continuous-time result is \textit{transferred} to the discrete setting. The $L$-Lipschitzness of the gradient is a key property to enable this transfer. In particular, the accelerated rate $(1-\sqrt{\mu/L})^k$ can be achieved if we take $s$ in \eqref{alg:nest_classic} large enough, namely $s = 1/L$.
This has been extensively used in the convex setting where similar convergence results are obtained for $\xcont$ and $\xdisc$, see $\textit{e.g.}$ \citep{suboydcandes,siegel2021accelerated,attouch2020firstorder,  shi2022understanding,aujdossrondPL,hyppo,li2024linear}. However, this approach may suffer from two drawbacks.

\noindent
\textbf{(i) Convergence properties do not transfer \:}  The convergence properties of $\xcont$ may not transfer to $\xdisc$ when $f$ is not convex. In the case of $L$-smooth functions, if the total energy decreases when applied to the solution of the continuous Heavy Ball equation, it is not the case for the discrete counterpart, see \citep[Section A.1]{jin2018accelerated}.  For functions with Lipschitz gradient and Hessian, a convergence result exists for the continuous Heavy Ball equation, but needs additional restrictions to deduce a discrete counterpart, see the first version of \citep{okamura2024primitive}.
In the case of strongly quasar-convex functions, in order to select a stepsize large enough to achieve the best known convergence rate with $\xdisctilde$ generated by \eqref{alg:nest_classic}, the non-convexity has to be restricted \citep{hermant2024study,guptanesterov}. Such restrictions are not necessary when considering $\xcont$ generated by \eqref{eq:nest_2_var}. Using side by side comparisons of discrete and continuous Lyapunov functions, Section 4.2 in \citep{hermant2024study} precises how non-convexity hurts the discretization process in this case.

\noindent
\textbf{(ii) Two analyses are needed \:}
Once a satisfactory convergence result is deduced for $(x_t)_{t\ge0}$, one has to adapt the proof to the deterministic setting. A new analysis has to be carried, typically involving more technical difficulties because we cannot differentiate anymore. 

\subsection{A Promising Alternative: Continuized Nesterov}\label{section:intro_continuized}

A modification of \eqref{eq:nest_2_var} is introduced in \citep{even2021continuized}, motivated by the study of asynchronous algorithms, and drawing on ideas originating from the analysis of Markov chains \citep{aldous2002reversible}. Intuitively, instead of making the gradient act continuously in $t \in \R_+$, gradient steps will occur at discrete
random times $T_1,T_2,\dots$ where for all $i$, $T_{i+1}-T_i$ are independent and follow an exponential law of parameter $1$. Precisely, between these random times, the dynamics of $\xz$ smoothly interpolate the two processes, namely 
\begin{align}\left\{
    \begin{array}{ll}
        d{x}_t &= \eta_t(z_t-x_t)dt  \\
        d{z}_t &= \eta'_t(x_t-z_t)dt
    \end{array}
\right.
\end{align}
At $t = T_i$, the process \textit{jumps} by performing a gradient step
\begin{equation}
    x_{T_i} = x_{T_i^{-}} - \gamma_{T_i^-}\nabla f(x_{T_i^-}), \quad z_{T_i} = z_{T_i^{-}} - \gamma'_{T_i^-}\nabla f(z_{T_i^-}),
\end{equation}
where $x_{T_i^-}$ denotes the right-limit $\lim_{t \to 0^-} x_{T_i + t}$.
This process can be written as a stochastic differential equation that we call the \textit{continuized Nesterov equations}
\begin{align}\label{eq:nest_continuized}\left\{
    \begin{array}{ll}
        dx_t &= \eta_t(z_t-x_t)dt - \gamma_{t^-} \nabla f(x_{t^-}) N(dt) \\
        dz_t &= \eta'_t(x_t-z_t)dt - \gamma'_{t^-} \nabla f(x_{t^-}) N(dt)
    \end{array}
\right.
\end{align}
where $N(dt) = \sum_{k\ge 0} \delta_{T_k}(dt)$ is a Poisson point measure. Intuitively, since $\E{dN_t} = dt$, the mean process $(\E{x_t}, \E{z_t})_{t \ge 0}$ is smooth, see Figure~\ref{fig:traj}. In particular, if $\nabla f$ is linear, it can be written exactly in the form of \eqref{eq:nest_2_var}. Consequently, individual realizations of the jump-driven process \eqref{eq:nest_continuized} provide a stochastic approximation of a much smoother expected trajectory. It implies that, for a given Lyapunov function $E_t := \varphi(x_t,z_t,t)$, the analysis of $\E{E_t}$ is close to that arising in the standard continuous setting, such as in the study of \eqref{eq:nest_2_var}.
\begin{figure}[h]
    \centering
 \includegraphics[scale=0.32]{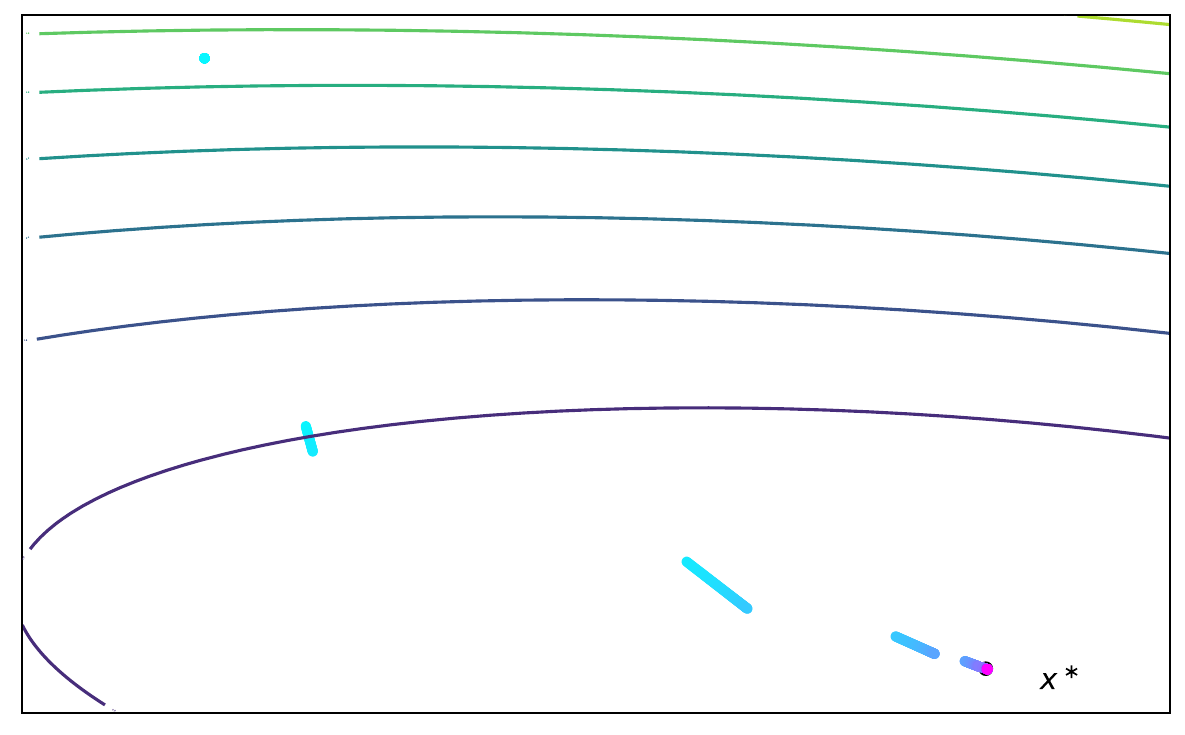}
  \includegraphics[scale=0.32]{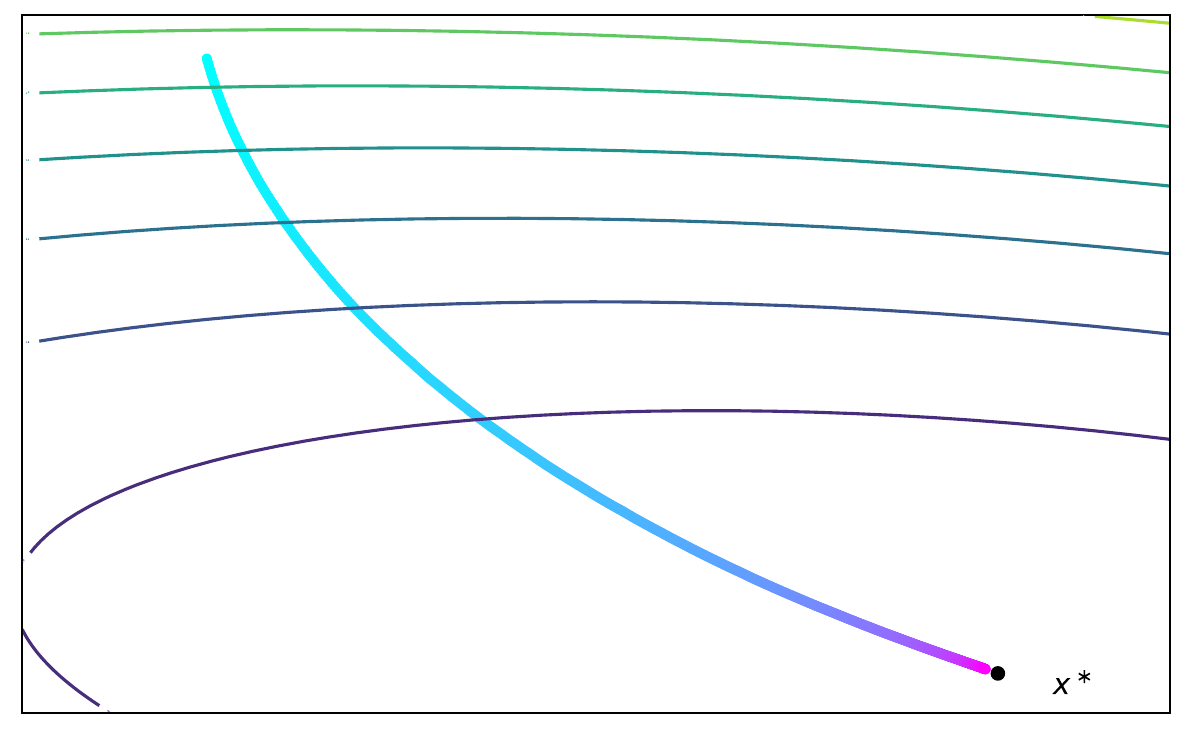}
    \includegraphics[scale=0.32]{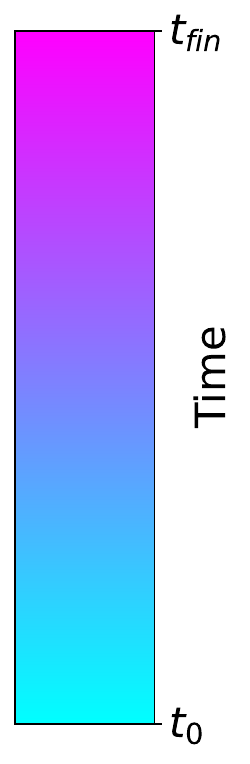}
    \caption{Left: Visualisation of a single trajectory of $\xcont$ satisfying \eqref{eq:nest_continuized}, between $t_0$ and $t_{\text{fin}}$. Right: Visualisation of the average over $10^5$ such trajectories, providing an estimate of the mean process $(\E{x_t})_{t \ge 0}$ for which the jumps are smoothed out. }
    \label{fig:traj}
\end{figure}

Strikingly, a core property of \eqref{eq:nest_continuized} is that the discretization $\tilde{x}_k := x_{T_{k}}$ and $\tilde{z}_k := z_{T_{k}}$, $k \in \N$, can be \textbf{computed exactly}. The resulting discretization writes as a Nesterov momentum algorithm of the form~\eqref{alg:nest_classic}, with stochastic parameters involving the random times $\{ T_k \}_{k\ge0}$. The convergence results obtained for the solution of \eqref{eq:nest_continuized} thus almost \textbf{automatically transfer} to the aforementioned discretization. It totally differs from the classic ODE framework,  where \eqref{alg:nest_classic} is a discrete \textbf{approximation} of \eqref{eq:nest_2_var}.
In their seminal work, \citep{even2021continuized} show that applied to (strongly) convex functions, the continuized version of the Nesterov momentum algorithm recovers, in expectation, similar results to the classical Nesterov algorithm. If this method has known some extension in the specific field of decentralized asynchronous algorithms
\citep{nabli2023dadao,nabli2023a2cid2}, its use in a more classic optimization setting remains limited. 

\noindent
\textbf{A non-convex application: strong quasar-convexity \:}
Strong quasar-convexity is a non-convex relaxation of strong convexity, that has received increasing attention in recent years \citep{hinder2020near,gower2021sgd,wang2023continuizedaccelerationquasarconvex,fu2023accelerated, alimisis2024characterization,hermant2024study,chenefficient2025,guptanesterov,lara2025delayed,de2025extending,farzin2025minimisation}. Compared to gradient descent, better convergence bounds exist for Nesterov momentum when minimizing these functions. The first of these accelerated results mixes the momentum mechanism with a subroutine that solves an optimization sub-problem \citep{quasarconvexold1,quasarconvexold2}, reduced to a line-search procedure which adds a logarithmic factor in the complexity in \citep{hinder2020near}. Avoiding such a subroutine requires to restrict the class of function, excluding functions that are too non-convex \citep{hermant2024study,guptanesterov}. The continuized method has been applied to strongly quasar-convex functions \citep{wang2023continuizedaccelerationquasarconvex}, recovering the convergence rate of the classical Nesterov momentum algorithm. Strikingly, it does so without requiring a line-search procedure or restricting the level of non-convexity. 
This indicates that the method is capable of achieving results that classical momentum algorithms have not been able to obtain so far, and we show in this work that it can go further.

\subsection{Limits of the Framework and Contributions}
Compared with the classical ODE framework, a key difference is that the continuized approach yields an exact discretization whose convergence properties inherit almost directly those of the continuous-time process.
Because of this conceptual difference, the continuized method may at first feel unfamiliar or non-intuitive to optimizers, which might explain why it is still largely unexploited. Moreover, it has some limitations: first, because of the inherent stochasticity of the continuized process, most of the existing results hold in expectation even if using the exact gradient. This leads to results that are weaker in nature than those of classical momentum algorithms, which provide deterministic guarantees. Results holding trajectory-wise with a given probability are given in \citep{wang2023continuizedaccelerationquasarconvex}, but the derived probability can be trivial for a large amount of the first iterations (Section~\ref{sec:high_proba}). Also, the current framework only allows to use smooth Lyapunov functions, preventing for example its application to non-smooth functions.  

\noindent
\textbf{Contributions \:}
(i) We present a clear methodology for employing the continuized method and extend the fundamental tools needed to make it operational. We illustrate the method on the class of smooth Polyak-\L{}ojasiewicz functions, deducing a new, though non-accelerated, convergence result, see Section \ref{sec:further}.

(ii) We provide an Itô formula (Proposition~\ref{prop:calc_sto_nondiff}) that allows to handle non-differentiable Lyapunov functions. In the case of exponentially fast convergence rates, we significantly tighten the existing bound on the probability to have a trajectory-wise convergence result, and further derive asymptotic almost sure convergence rates.

(iii)  Adapting results coming from the classic ODE framework literature, we improve the best known convergence rate by a constant factor for smooth strongly quasar-convex functions. We further relax the existing assumption on the set of minimizers for this class of functions. Our results hold when considering unbiased stochastic estimates of the gradients, as long as a strong growth condition is verified. 
\section{Assumptions and Notations}

We denote $\minset := \arg  \min_{\R^d} f(x)$, $x^\ast$ any element of $ \minset$ and $f^\ast := \min_{x\in \R^d} f(x)$. We denote $\Omega$ the underlying probability space over which any process and random variable of this paper is defined, and $\mathcal{P}(\Xi)$ some law over a metric space $\Xi$.
For a random variable $X$ and any measurable function $\psi: \Xi \to \R$,~ we denote $\mathbb{E}_{\xi}[\psi(X,\xi)] := \int_\Xi \psi(X,\xi)\mathcal{P}(d\xi)$, which is the conditional expectation of $\psi(X,\xi)$ with respect to $X$.
$\mathcal{E}(1)$ denotes the exponential distribution with parameter $1$, and $\Gamma(k,1)$ the Gamma distribution with shape parameter $k$ and rate $1$. For $a,b \in \R$, $a \wedge b = \min \{a,b \}$.
\subsection{Geometrical Assumptions}\label{sec:gem_ass}
We introduce the geometrical assumptions used through the paper.
\begin{assumption}[L-smooth]\label{ass:l_smooth}
    $f$ is such that $\forall x,y\in \R^d,
 ~ f(x)-f(y)-\dotprod{\nabla f(y),x-y}  \le \frac{L}{2}\norm{x-y}^2$, for some $L>0$.
\end{assumption}
The L-smooth property is fundamental in many analyses of gradient-based methods, including momentum algorithms. Intuitively, it gives an upper bound on the curvature of the function. In particular for $C^2$ functions, Assumption~\hyperref[ass:l_smooth]{(L-smooth)} is equivalent to having the eigenvalues of the Hessian matrix uniformly upper-bounded by $L$. Functions satisfying Assumption~\hyperref[ass:l_smooth]{(L-smooth)} can have many local minima or saddle points, such that finding a global minimum is in general intractable \citep[Paragraph 2.2]{danilova2022recent}. The Polyak-\L{}ojasiewicz (PL) inequality \citep{Polyak1963,lojasiewicz1963topological} ensures to avoid such situations. 
\begin{assumption}[PL$_\mu$]\label{ass:pl}
    $f$ is such that $ \forall x \in \R^d$, $\norm{\nabla f(x)}^2 \ge 2\mu(f(x)-f^\ast)$, for some $\mu>0$. 
\end{assumption}
It is immediate to check that under Assumption~\hyperref[ass:pl]{(PL$_\mu$)}, a critical point is also a global minimizer. It also ensures that $f$ has a \textit{quadratic growth}, namely $f(x)-f^\ast \ge 0.5\mu \norm{x-x^\ast}^2$, for all $x\in\R^d$ \citep{bolte2017error,karimi2016linear}.
Gradient descent is known to make functions satisfying Assumptions~\hyperref[ass:l_smooth]{(L-smooth)} and \hyperref[ass:pl]{(PL$_\mu$)} converge linearly to their minimum, with a rate of order $\bigO((1-\mu/L)^k$) after $k$ iterations \citep{bolte2017error, karimi2016linear}. Importantly, momentum algorithms cannot significantly improve over this bound \citep{PLlowerbound}. In particular, the dependence on $\mu/L$ cannot be improved to $\sqrt{\mu/L}$ as in the strongly convex setting.  
To achieve such an acceleration, one can consider $(\tau,\mu)$-strong quasar-convexity.
\begin{assumption}[SQC$_{\tau,\mu}$]\label{ass:sqc}
  $f$ verifies
    $f^\ast \ge f(x) + \frac{1}{\tau}\dotprod{\nabla f(x),x^\ast-x} + \frac{\mu}{2}\norm{x-x^\ast}^2$ for all $x\in \mathbb{R}^d$, and some $(\tau,\mu)\in (0,1]\times \mathbb{R}^\ast_+$.
\end{assumption}

\noindent
Assumption~\hyperref[ass:sqc]{(SQC$_{\tau,\mu}$)} defines a non-convex relaxation of strong convexity that may exhibit important oscillations, see Figure~\ref{fig:figure_sqc}, while being a stronger assumption than Assumption~\hyperref[ass:pl]{(PL$_\mu$)} \citep[Proposition 1]{hermant2024study}.  Functions satisfying Assumptions~\hyperref[ass:l_smooth]{(L-smooth)} and \hyperref[ass:sqc]{(SQC$_{\tau,\mu}$)} achieve similar rates as strongly convex functions, namely $\bigO((1-\tau \mu/L)$ for gradient descent and $\bigO((1-\tau \sqrt{\mu/L})$ when using momentum, see Section~\ref{sec:improv_gen}. One possible application of these functions is the training of generalized linear models \citep[Section 3]{wang2023continuizedaccelerationquasarconvex}. However the unique minimizer property induced by Assumption~\hyperref[ass:sqc]{(SQC$_{\tau,\mu}$)} \citep[Observation 4] {hinder2020near} prevents several machine learning problems to directly fit in. 
Nonetheless, replacing $x^\ast$ in Assumption~\hyperref[ass:sqc]{(SQC$_{\tau,\mu}$)} by the projection of $x$ onto the minimizers of $f$, it becomes a realistic assumption in some over-parameterized deep learning context \citep[Section 2.3]{guptanesterov}. Moreover, empirical studies showed that fixing $x^\ast$ to be the converging point of a first order algorithm, close assumptions are verified along the minimization path of some neural network \citep{zhou2019sgd,guille2022gradient}. We will further relax this uniqueness property in Section~\ref{sec:proj}.
\begin{figure}[ht]
    \centering
    \begin{minipage}[b]{0.4\textwidth}
        \centering
        \includegraphics[width=\textwidth]{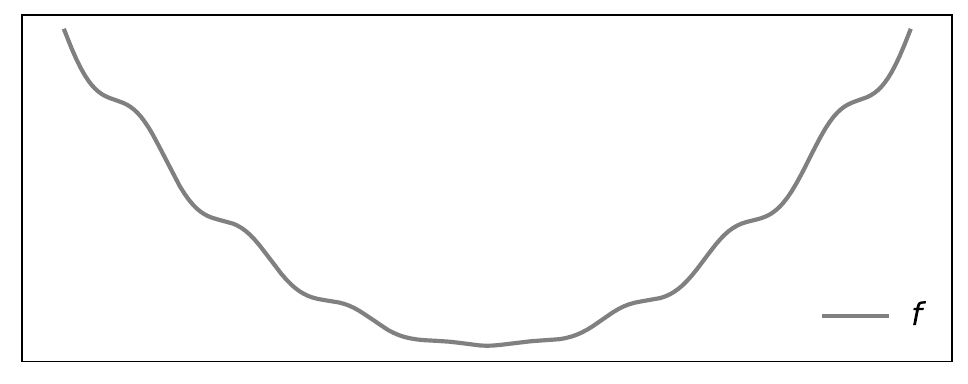}\\[1ex]
        \includegraphics[width=\textwidth]{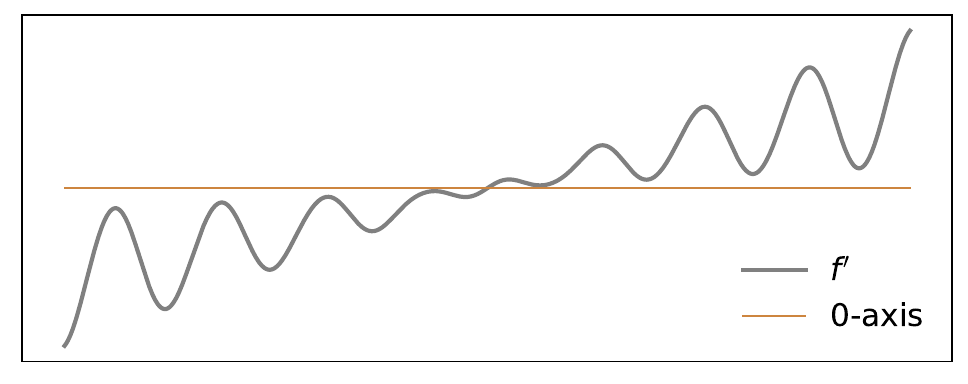}
    \end{minipage}%
    \begin{minipage}[b]{0.38\textwidth}
        \centering
        \includegraphics[width=\textwidth]{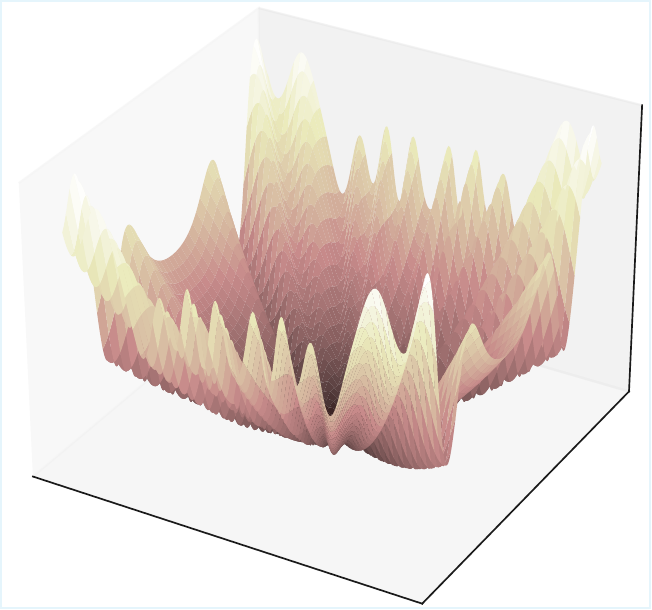}
    \end{minipage}
    \caption{Visualization of functions verifying Assumption~\hyperref[ass:sqc]{(SQC$_{\tau,\mu}$)}. Left: 1-dimensional example $f(t) =  0.5\cdot(t +0.15\sin(5t))^2$ on top, its derivative on the bottom. Right: 2-dimensional example of the form $h(x,y) = f(\norm{(x,y)})g\bpar{\frac{(x,y)}{\norm{(x,y)}}}$, where $f$ verify  Assumption~\hyperref[ass:sqc]{(SQC$_{\tau,\mu}$)} and $g \ge 1$. See details in Appendix~\ref{app:fig_detail}. It clearly appears in both cases that this class of functions permits high oscillations.}
    \label{fig:figure_sqc}
\end{figure}
\subsection{Stochastic Gradients}
We consider using stochastic estimators of the gradient $\nabla f$.
\begin{assumption}[SGC$_\rho$]\label{ass:sgc}
    We assume we have access to stochastic approximations $\nabla f(x,\xi)$ of the real gradient $\nabla f(x)$, where $\xi \sim \mathcal{P}(\Xi)$, such that
    \begin{enumerate}[label=(\roman*)]
        \item (Unbiased estimator) $\forall x \in \R^d, ~ \E{\nabla f(x,\xi)} = \nabla f(x)$. 
        \item (Strong Growth Condition) $\exists \rho \ge 1, \forall x \in \R^d,~ \mathbb{E}\left[ \norm{\nabla f(x,\xi)}^2 \right] \le \rho \norm{\nabla f(x)}^2$. 
    \end{enumerate} 
       
\end{assumption}
Assumption~\hyperref[ass:sgc]{(SGC$_\rho$)}-(ii) is sometimes named \textit{multiplicative noise}. It induces that a point $\hat
x$ such that $\nabla f(\hat{x}) = 0$ also verifies $\nabla f(\hat{x},\xi) = 0$,  $\forall \xi \in \Xi$. In a machine learning context, this typically occurs in the overparameterized regime \citep{cooper2018,allen2019convergence, nakkiran2021deep, zhang2021understanding}. Under this assumption, because the variance of the estimator vanishes when approaching critical points, we obtain significantly better convergence rate, see \citep{ma2018power,gower2019sgd,gower2021sgd} for the case of gradient descent, \citep{vaswani2019fast} for the case of Nesterov momentum. Note that Assumption~\hyperref[ass:sgc]{(SGC$_\rho$)} encompasses the deterministic gradient setting as a special case, corresponding to $\rho = 1$ in (ii).
\begin{example}[Finite-Sum Setting]
    Assume $f := \frac{1}{N}\sum_{i=1}^N f_i$, where for each $i \in \{1,\dots,N\}$, $f_i :\R^d \to \R$ is $C^1$. Let $\xi$ drawn uniformly in $\{1,\dots,N\}$, and define $\nabla f(x,\xi) = \nabla f_{\xi}(x)$. Then, Assumption~\hyperref[ass:sgc]{(SGC$_\rho$)} becomes
         $$\exists \rho \ge 1, \forall x \in \R^d,~ \frac{1}{N}\sum_{i=1}^N \norm{\nabla f_i(x)}^2 \le \rho \norm{\nabla f(x)}^2.$$ 
         Note that (i) is automatically verified as $\forall x \in \R^d, ~ \E{\nabla f_{\xi}(x))} = \frac{1}{N}\sum_{i=1}^N \nabla f(x) = \nabla f(x)$. 
\end{example}
In the finite-sum setting, \citep{hermantgradient} relates Assumption~\hyperref[ass:sgc]{(SGC$_\rho$)} to the averaged correlation of pairwise gradients. They also
provide experiments which show that it is verified along the path of several first-order algorithms when optimizing some neural networks. From an optimization point of view, under Assumption~\hyperref[ass:l_smooth]{(L-smooth)}, one can deduce a stochastic version of the descent lemma, whose proof is deferred in Appendix~\ref{app:descent_lemma}.
\begin{lemma}\label{lem:descent_sgc}
Under Assumption~\hyperref[ass:l_smooth]{(L-smooth)} and \hyperref[ass:sgc]{(SGC$_\rho$)}, if $\gamma \le \frac{1}{L \rho}$, we have

\begin{equation*}
    \forall x \in \R^d,~ \mathbb{E}\left[ f(x-\gamma \nabla f(x,\xi)) - f(x) \right] \leq \gamma\left(\frac{L \rho}{2}\gamma - 1 \right) \lVert  \nabla f(x) \rVert^2.
\end{equation*}
\end{lemma}
Lemma~\ref{lem:descent_sgc} bounds, in expectation, the decrease of the function $f$ when performing a stochastic gradient step. For deterministic gradient, \textit{i.e.} $\rho = 1$, we recover the classical descent lemma \citep[eq. (1.2.19)]{nesterovbook}.
\section{How to Use the Continuized Nesterov Framework}\label{sec:further}
In this section, we detail the methodology for leveraging the continuized Nesterov dynamics to derive convergence results. 
It can be outlined through the three following steps: \textit{(i) }continuous-time analysis via stochastic Lyapunov functions, \textit{(ii)} derivation of an exact discrete algorithm, and \textit{(iii)} automatic transfer of convergence guarantees to this algorithm, using stopping-time arguments.

First, we note that the continuized equations naturally handle stochastic gradients. Following the discussion of Section~\ref{section:intro_continuized}, consider the same random times $\{ T_k \}_{k\in \N}$. Assume $\xi_1,\xi_2,\dots$ are i.i.d random variables of law $\mathcal{P}$.  At $t = T_k$, a stochastic gradient step is performed
\begin{equation*}
    x_{T_k} = x_{T_k^{-}} - \gamma\nabla f(x_{T_k^-},\xi_{k}), \quad z_{T_k} = z_{T_k^{-}} - \gamma'\nabla f(z_{T_k^-},\xi_{k}).
\end{equation*}
Between these random times, the dynamics followed by $\xz$ is still
\begin{align*}\left\{
    \begin{array}{ll}
        d{x}_t &= \eta(z_t-x_t)dt  \\
        d{z}_t &= \eta'(x_t-z_t)dt
    \end{array}
\right.
\end{align*}
These \textit{continuized Nesterov equations} write as ((15)-(16) in \citep{even2021continuized})
\begin{align}\label{eq:nest_continuized_sto}\tag{CNE}\left\{
    \begin{array}{ll}
        dx_t &= \eta(z_t-x_t)dt - \gamma \int_{\Xi}\nabla f(x_{t^-},\xi) N(dt,d\xi) \\
        dz_t &= \eta'(x_t-z_t)dt - \gamma' \int_{\Xi}\nabla f(x_{t^-},\xi) N(dt,d\xi)
    \end{array}
\right.
\end{align}
with the Poisson point measure $N(dt,d\xi) = \sum_{k\ge 0} \delta_{(T_k,\xi_k)}(dt,d\xi)$ on $\R_+ \times  \Xi$, which has intensity $dt \otimes \mathcal{P}$. By definition, the process is càdlàg, \textit{i.e.} right-continuous with left limits. We assume $x_0 = z_0$ is deterministic. Note that compared to the discussion of Section~\ref{section:intro_continuized}, we consider a constant set of parameters $(\eta,\eta',\gamma,\gamma')$.
\begin{remark}
Because gradient steps intervene at discrete times, the continuized Nesterov equations are highly flexible to modifications. By instance, we could consider replacing gradient step by proximal updates \citep{rockafellar1976monotone}, or using different stochastic estimators of the gradients, \textit{e.g.} using variance reduction techniques \citep{johnson2013accelerating}. Although we believe these extensions are interesting directions for future research, they fall outside the scope of the present work.
\end{remark}
Our main interest is to analyze the convergence properties of $\xz$.
In the following section, we provide an analytical tool that replaces differentiation in the presence of jumps.
\subsection{Lyapunov Analysis with Stochastic Calculus}\label{sec:ito_lyap}

To deduce convergence results using \eqref{eq:nest_continuized_sto}, a classical strategy is to consider a well chosen Lyapunov function,\textit{ e.g.} $t \to \E{ f(x_t)}$, and to show it is decreasing. Although $\xcont$ is indexed by continuous time, it cannot be differentiated with respect to time because of the jumps introduced by the Poisson component. Fortunately, the theory of stochastic calculus provides analogous tools, such that we can use a Lyapunov-style approach.
\begin{proposition}[$C^1$-Itô Formula]\label{prop:sto_calc_abridged}
Let $x_t \in \R^d$ be a solution of 
\begin{equation*}
    dx_t = \zeta(x_t)dt + \int_\Xi G(x_{t^-},\xi)N(dt,d\xi)
\end{equation*}
where $\zeta : \R^d \to \R^d$ is a locally Lipschitz function and $G : \R^d \times \Xi \to \R^d$ is measurable. Let $\varphi : \R^d \to \R$ be with Lipschitz gradient. Then,
\begin{equation}\label{eq:prop_calc_sto_core_abbriged}
    \varphi(x_t) = \varphi(x_0) + \int_{0}^t \dotprod{\nabla \varphi(x_s),\zeta(x_s)}ds + \int_{[0,t]}  \mathbb{E}_\xi[ \varphi(x_s + G(x_s,\xi)) - \varphi(x_s)]ds + M_t,
\end{equation}

where $M_t$ is a martingale such that $\E{M_t}=0$, $\forall t\ge 0$.

\end{proposition}
This Itô formula is a close formulation of \cite[Proposition 2]{even2021continuized}. For readers who are not familiar with martingales or stochastic differential equations, standard references include \citep{williams1991probability,protter2012stochastic}.
Proposition~\ref{prop:sto_calc_abridged} can be applied to a well chosen Lyapunov function $\varphi$. In particular, it suffices to show that $\dotprod{\nabla \varphi(x_s),\zeta(x_s)} + \mathbb{E}_\xi{ \varphi(x_s + G(x_s,\xi)) - \varphi(x_s)} \le 0$ to have that $t \to \varphi(x_t)$ is a supermartingale, which in turn induces
$$\E{\varphi(x_t)} \le \E{\varphi(x_0)}.$$
 We further relax this $C^1$ condition in Proposition~\ref{prop:calc_sto_nondiff}.
\begin{remark}
    The $ \int_{0}^t \dotprod{\nabla \varphi(x_s),\zeta(x_s)}ds$ term appearing in \eqref{eq:prop_calc_sto_core_abbriged} is purely of continuous nature. Computations involving this term will be similar to the one involving the classic deterministic system \eqref{eq:nest_2_var}. The term $ \int_{[0,t]}\mathbb{E}_\xi [\varphi(x_s + G(x_s,\xi)) - \varphi(x_s)]ds$ is induced by the Poisson process, and will involve some terms leading to computations that are closer to the one occurring when analyzing \eqref{alg:nest_classic}. As mentioned in \citep{even2021continuized}, the continuized framework involves a continuous-discrete mix.
\end{remark}

\subsection{Almost Automatic Discretization}\label{sec:automatic_transfer}
Because the solution to \eqref{eq:nest_2_var} requires knowledge of the gradient at all times $t \in \R_+$, it is generally not available in closed form. Thus, convergence guarantees for the continuous system cannot be applied directly in practice, and one typically proceeds by establishing a corresponding result for a discretization, \textit{e.g.} \eqref{alg:nest_classic}. 
A remarkable property of \eqref{eq:nest_continuized_sto} is that the gradient acts only at discrete times. Defining $\tilde{y}_k := x_{T_{k+1}^-}$, $\tilde{x}_k := x_{T_{k}}$ and $\tilde{z}_k := z_{T_{k}}$ as evaluations of the process $\xz$ defined by \eqref{eq:nest_continuized_sto}, one can explicitly integrate the trajectory between each times $\{ T_k \}_{k \in \N}$.

\begin{proposition}\label{prop:discretization_constant_param}
Let $\xz$ follow \eqref{eq:nest_continuized_sto} with $\eta + \eta' > 0$ and with underlying jumping times $\tk$. Define $\tilde{y}_k := x_{T_{k+1}^-}$, $\tilde{x}_{k+1} := x_{T_{k+1}}$ and $\tilde{z}_{k+1} := z_{T_{k+1}}$ as evaluations of this process. Then, $(\tilde{y}_k,\tilde{x}_k, \tilde{z}_k)$ writes as a Nesterov algorithm in the form of \eqref{alg:nest_classic} with stochastic parameters 
\begin{align}\label{alg:constant_param}\tag{CNA}\left\{
    \begin{array}{ll}
        \tilde{y}_k &= (\tilde{z}_{k} - \tilde{x}_k)\frac{\eta}{\eta+\eta'}\bpar{1 -e^{-(\eta + \eta')(T_{k+1}-T_k)} } + \tilde{x}_k \\
        \tilde{x}_{k+1} &=  \tilde{y}_k  - \gamma \nabla f(\tilde{y}_k,\xi_{k+1})\\
        \tilde{z}_{k+1} &=   \tilde{z}_{k} + \eta' \frac{\bpar{1- e^{-(\eta+\eta')(T_{k+1}-T_k)}}}{\eta' + \eta e^{-(\eta+\eta')(T_{k+1}-T_k)}}(\tilde{y}_k - \tilde{z}_{k}) - \gamma' \nabla f(\tilde{y}_k,\xi_{k+1})
    \end{array}
\right.
\end{align}
\end{proposition}
Proposition~\ref{prop:discretization_constant_param} is a generalization of \citep{even2021continuized, wang2023continuizedaccelerationquasarconvex} in the case of general constant parameters, see its proof in Appendix~\ref{app:discretization}. Note that the random times $\{ T_k \}_{k \in \N}$ can be easily simulated, such that the algorithm is implementable. We also note that $\eta' + \eta e^{-(\eta+\eta')(T_{k+1}-T_k)} \neq 0$, almost surely, which prevents from definition problems in practice. We refer to \eqref{alg:constant_param} as the \textit{continuized Nesterov algorithm}. Compared to the classic ODE framework described in Section~\ref{sec:classic_ode_framework}, \eqref{alg:constant_param} is not an approximation of a continuous counterpart, it is an exact evaluation of a realization of \eqref{eq:nest_continuized_sto}. As $\tilde{x}_{k}$ is exactly $x_{t}$ evaluated at the stopping time $T_k$, one can transfer a convergence result to the sequence $\{\tilde{x}_k \}_{k\in\N}$, which we formalize as follows.
\begin{theorem}[Stopping theorem]\label{thm:martingal_stopping}
    Let $(\varphi_t)_{t\in \R_+}$ be a non-negative process with càdlàg trajectories, such that it verifies
    $$  \varphi_t \le K_0 + M_t, $$
    for some positive random variable $K_0$, some martingale $\mart$ with $\E{M_0} = 0$. Then, for any almost surely finite stopping time $\tau$, one has 
$$ \E{\varphi_\tau} \le \E{K_0}.$$
\end{theorem}
\begin{proof}
Let $\tau$ be an almost surely finite time, \textit{i.e.} $ \tau(\omega)< + \infty$ almost surely.
The stopped martingale $\{ M_{t \wedge \tau} \}_{t \in \R_+}$ is a martingale, with $\E{M_{t \wedge \tau}} = \E{M_{0 \wedge \tau}}  = 0$ for all $t \in \R_+$.
    Then, the stopped process $V_t := \varphi_{\tau \wedge t}$ is such that
    \begin{equation*}
        \forall t\ge0,~ \E{V_{t}} \le \E{K_0}.
    \end{equation*}
    As we assume $\inf_t \varphi_t \ge 0$, we also have $\inf_t  V_t \ge 0$. We thus apply Fatou's Lemma
    \begin{equation}\label{eq:stopping_thm_0}
        \E{  \liminf_{t \to +\infty} \:V_t}  \le \liminf_{t \to +\infty} \: \E{V_t} \le \E{K_0}.
    \end{equation}
    $\tau$ being almost surely finite implies that $\lim_{t \to +\infty} t \wedge \tau(\omega) = \tau(\omega) $, almost surely. Then, for almost all $\omega \in \Omega$, we have
    \begin{equation}\label{eq:stopping_thm_1}
        \lim_{t \to +\infty} V_t(\omega) =  \lim_{t \to +\infty} \varphi_{t\wedge \tau(\omega)}(\omega) = \varphi_{\tau(\omega)}(\omega), 
    \end{equation}
    namely $V_t$ converges almost surely to $\varphi_{\tau}$. Plugging \eqref{eq:stopping_thm_1} into \eqref{eq:stopping_thm_0}, we obtain
    $$ \E{\varphi_\tau} \le \E{K_0}.$$
\end{proof}
The stopping theorem used in \citep{even2021continuized} assumes the process is a uniformly integrable super-martingale, which might be difficult to verify. We state a result that holds for a broader range of $\{\varphi_t \}_{t\in\R_+}$, and is more convenient to use. Finally, we note that this automatic transfer from the continuous to the discrete setting comes at the cost of yielding non-deterministic convergence results, even in the presence of deterministic gradients.

\subsection{Application to Smooth, Polyak-\L{}ojasiewicz Functions}\label{sec:illustr_pl}
We illustrate the previously presented methodology in the context of minimizing $L$-smooth, $\mu$-Polyak-\L{}ojasiewicz functions. We start by deriving a convergence result associated to the continuous process $\xcont$, using the Itô formula introduced in Section~\ref{sec:ito_lyap}.
 \begin{proposition}\label{prop:pl_cont}
If $\xz$ solution of (\ref{eq:nest_continuized_sto}) with $\gamma \le \frac{1}{\rho L}$, $\gamma' = \gamma + \sqrt{\frac{\gamma}{2 \rho}}$, $\eta = \sqrt{\frac{\gamma}{2\rho}}$ and $\eta' = \frac{\mu \gamma}{2} - \sqrt{\frac{\gamma}{2\rho}}$ under  Assumptions~\hyperref[ass:l_smooth]{(L-smooth)},~\hyperref[ass:pl]{(PL$_\mu$)} and \hyperref[ass:sgc]{(SGC$_\rho$)}, we have
\begin{equation*}
   \E{f(x_t)-f^\ast} \le e^{-\frac{\mu\gamma}{2}t}(f(x_0)-f^\ast).
\end{equation*}
 \end{proposition}
\begin{proof}

\textbf{Formalism Setting \:} We start to set the formalism such that we can apply Proposition~\ref{prop:sto_calc_abridged}.
Let $\overline{x}_t= (t,x_t,z_t)$ where $\xz$ satisfy \eqref{eq:nest_continuized_sto}. It satisfies $d\overline{x}_t = \zeta(\overline{x}_t)dt + G(\overline{x}_t,\xi)N(dt,d\xi)$, with
\begin{equation*}
    \zeta(\overline{x}_t) = \begin{pmatrix}
    1 \\ \eta(z_t-x_t) \\ \eta'(x_t-z_t)
    \end{pmatrix},\quad G(\overline{x}_t,\xi) = \begin{pmatrix}
    0 \\ -\gamma \nabla f(x_t,\xi) \\ -\gamma' \nabla f(x_t,\xi).
    \end{pmatrix}
\end{equation*}
We apply Proposition~\ref{prop:sto_calc_abridged} to $\varphi(\overline{x}_t)$, where
\begin{equation}
    \varphi(t,x,z) = a_t(f(x)-f^\ast) + \frac{b_t}{2}\norm{x-z}^2,
\end{equation}
such that for a martingale $M_t$, we have
\begin{equation}\label{eq:prop_calc_sto_lsmooth}
    \varphi(\overline{x}_t) = \varphi(\overline{x}_0) + \int_{0}^t \dotprod{\nabla \varphi(\overline{x}_s),\zeta(\overline{x}_s)}+ \mathbb{E}_\xi[ \varphi(\overline{x}_s + G(\overline{x}_s,\xi)) - \varphi(\overline{x}_s)]ds + M_t.
\end{equation}

\noindent
\textbf{Computations \:} 
We compute $\dotprod{\nabla \varphi(\overline{x}_s),\zeta(\overline{x}_s)} $ and $\mathbb{E}_{\xi} \varphi(\overline{x}_s+ G(\overline{x}_s)) - \varphi(\overline{x}_s)$.
We have
\begin{equation*}
    \frac{\partial \varphi}{\partial s}   = \frac{da_s}{ds}(f(x)-f^\ast) + \frac{1}{2}\frac{db_s}{ds}\norm{x-z}^2, \quad \frac{\partial \varphi}{\partial x} = a_s\nabla f(x) + b_s(x-z),\quad  \frac{\partial \varphi}{\partial z} = b_s(z - x).
\end{equation*}
So, 
\begin{align}
    &\dotprod{\nabla \varphi(\overline{x}_s),\zeta(\overline{x}_s)}\nonumber \\
    &= \frac{da_s}{ds}(f(x_s)-f^\ast) + \frac{1}{2}\frac{db_s}{ds}\norm{x_s-z_s}^2\nonumber \\
    &+\dotprod{a_s\nabla f(x_s) + b_s(x_s - z_s),\eta(z_s-x_s)} + \dotprod{b_s(z_s-x_s),\eta'(x_s - z_s)} \nonumber\\
    &= \frac{da_s}{ds}(f(x_s)-f^\ast) + \frac{1}{2}\frac{db_s}{ds}\norm{x_s-z_s}^2+ a_s\eta \dotprod{\nabla f(x_s),z_s-x_s} - b_s\bpar{\eta + \eta'}\norm{x_s-z_s}^2 \label{eq:lem_l_smooth:1}.
\end{align}
Also, 
\begin{align}
   &\mathbb{E}_{\xi}[ \varphi(\overline{x}_s+ G(\overline{x}_s)) - \varphi(\overline{x}_s) ]\nonumber \\
   &= a_s\mathbb{E}_{\xi}[f(x_s - \gamma \nabla f(x_s,\xi)) - f(x_s)]\nonumber\\
   &+ \frac{b_s}{2}\mathbb{E}_{\xi}[\bpar{\norm{x_s - \gamma \nabla f(x_s,\xi) - (z_s - \gamma'\nabla f(x_s,\xi)}^2 - \norm{x_s-z_s}^2}]\nonumber\\
    &= a_s\mathbb{E}_{\xi}[f(x_s - \gamma \nabla f(x_s,\xi)) - f(x_s)]  + b_s\frac{(\gamma' - \gamma)^2}{2}\mathbb{E}_{\xi}[\norm{\nabla f(x_s,\xi)}^2] \nonumber\\
    &+ b_s(\gamma' - \gamma)\dotprod{\mathbb{E}_{\xi}[\nabla f(x_s,\xi)],x_s - z_s}\nonumber\\
    &\le \bpar{b_s \rho(\gamma' - \gamma)^2 - a_s\gamma(2-L \gamma \rho)}\frac{1}{2}\norm{\nabla f(x_s)}^2 + b_s(\gamma' - \gamma)\dotprod{\nabla f(x_s),x_s - z_s}.\label{eq:lem_l_smooth:2}
\end{align}
The last inequality uses Assumption~\hyperref[ass:sgc]{(SGC$_\rho$)} and Lemma~\ref{lem:descent_sgc}.
We combine \eqref{eq:lem_l_smooth:1} and \eqref{eq:lem_l_smooth:2}
\begin{eqnarray}
    \begin{aligned}\label{eq:lem_l_smooth:3}
       & \dotprod{\nabla \varphi(\overline{x}_s),\zeta(\overline{x}_s)}  +    \mathbb{E}_{\xi}[ \varphi(\overline{x}_s+ G(\overline{x}_s)) - \varphi(\overline{x}_s) ]
       \\&\le  \frac{da_s}{ds}(f(x_s)-f^\ast) + \frac{1}{2}\frac{db_s}{ds}\norm{x_s-z_s}^2+(a_s\eta - b_s(\gamma' - \gamma))\dotprod{\nabla f(x_s),z_s-x_s} \\
        &- b_s(\eta + \eta')\norm{x_s - z_s}^2 +  \bpar{b_s\rho(\gamma' - \gamma)^2 - a_s\gamma(2-L \gamma\rho)}\frac{1}{2}\norm{\nabla f(x_s)}^2 .
    \end{aligned}
\end{eqnarray}
Now, using Assumption~\hyperref[ass:pl]{(PL$_\mu$)} in \eqref{eq:lem_l_smooth:3} and rearranging, we deduce
\begin{eqnarray}
    \begin{aligned}\label{eq:lem_l_smooth:4}
        &\dotprod{\nabla \varphi(\overline{x}_s),\zeta(\overline{x}_s)}  +    \mathbb{E}_{\xi}[ \varphi(\overline{x}_s+ G(\overline{x}_s)) - \varphi(\overline{x}_s) ]\\
        &\le  \bpar{\frac{da_s}{ds}+ \mu \bpar{b_s\rho(\gamma' - \gamma)^2 - a_s\gamma(2-L \gamma\rho)} }(f(x_s)-f^\ast) \\
        &+ \bpar{\frac{1}{2}\frac{db_s}{ds} - b_s (\eta + \eta')}\norm{x_s-z_s}^2 + (a_s\eta - b_s(\gamma' - \gamma))\dotprod{\nabla f(x_s),z_s-x_s}.
    \end{aligned}
\end{eqnarray}

\noindent
\textbf{Parameter Tuning \:}
As it is standard in many convergence proofs, we now choose the parameters driving equation \eqref{eq:nest_continuized_sto} such that we get to the desired result. We set $a_s = b_s = e^{\theta s}$, for some $\theta>0$. 
Fixing $ \gamma' = \gamma + \sqrt{\frac{\gamma}{2\rho}}$ ensures 
$$\rho(\gamma' - \gamma)^2 - \gamma(2-L \gamma\rho) \le \frac{\gamma}{2} - \gamma = -\frac{\gamma}{2}.$$ The latter holds because $\gamma \le \frac{1}{\rho L} \Rightarrow -\gamma(2-L \rho \gamma) \le - \gamma$. Setting $\eta =  \gamma' - \gamma = \sqrt{\frac{\gamma}{2\rho}} $ cancels the scalar product in \eqref{eq:lem_l_smooth:4}.  We thus have
\begin{equation}
\begin{aligned}\label{eq:L_smooth_eq1}
       & \dotprod{\nabla \varphi(\overline{x}_s),\zeta(\overline{x}_s)}  +    \mathbb{E}_{\xi}\left[\varphi(\overline{x}_s+ G(\overline{x}_s)) - \varphi(\overline{x}_s)\right] \\
     &\le  e^{\theta s}\bpar{\frac{\theta }{2}-(\eta + \eta')}\norm{x_s - z_s}^2 + e^{\theta s} \bpar{\theta - \mu\frac{\gamma}{2}}(f(x_s)-f^\ast).
\end{aligned}
\end{equation}
Fixing $\theta = \frac{\mu \gamma}{2}$ and $\eta' = \frac{\mu \gamma}{2} -\eta =  \frac{\mu \gamma}{2} - \sqrt{\frac{\gamma}{2\rho}}  $ cancels the right hand side of \eqref{eq:L_smooth_eq1}.

\noindent
\textbf{Conclusion \:}
From \eqref{eq:prop_calc_sto_lsmooth} and the previous computations,
we get
\begin{align}
     &\varphi(\overline{x}_t) \le \varphi(\overline{x}_0)  + M_{t}  \label{eq:l_smooth_lyap_dec}.
     \end{align}
     Using the definition of $\varphi$, $x_0 = z_0$ and  ignoring the $e^{\frac{\mu \gamma}{2}t}\norm{x_t-z_t}^2$ term, we get
     \begin{equation}\label{eq:pl_cont}
   e^{\frac{\mu\gamma}{2}t}(f(x_t)-f^\ast) \le (f(x_0)-f^\ast) + M_t.
\end{equation}
$M_t$ is a martingale with $\E{M_0} = 0$, such that $\mathbb{E}[M_t] = \mathbb{E}[M_0] = 0$, see Proposition~\ref{prop:sto_calc_abridged}. So, together with the fact that $x_0$ is assumed to be deterministic, taking expectation in \eqref{eq:pl_cont}, we obtain  
\begin{equation*}
   \E{f(x_t)-f^\ast} \le e^{-\frac{\mu\gamma}{2}t}(f(x_0)-f^\ast).
\end{equation*}

\end{proof}
The only technical tool we use is a stochastic descent Lemma which controls the decrease of the function when performing a gradient step, see Lemma~\ref{lem:descent_sgc}. Note that such descent lemmas typically arise in discrete analyses, rather than in the study of continuous processes. It is used in our proof to handle a term induced by the discrete component of the process. Before commenting on the result of  Proposition~\ref{prop:pl_cont}, we leverage the tools introduced in Section~\ref{sec:automatic_transfer} to transfer its result, applied to $\xcont$ in \eqref{eq:nest_continuized_sto}, to a result on $\xdisctilde$ in \eqref{alg:constant_param}.
\begin{theorem}\label{thm:cv_pl}
Assume $T_1,\dots,T_k$ are random variables such that $T_{i+1} - T_i$ are  i.i.d  of law $\mathcal{E}(1)$, with convention $T_0 = 0$. Under Assumption~\hyperref[ass:l_smooth]{(L-smooth)},~\hyperref[ass:pl]{(PL$_\mu$)} and \hyperref[ass:sgc]{(SGC$_\rho$)}, iterations of (\ref{alg:constant_param}) with $\gamma \le \frac{1}{\rho L}$, $\gamma' = \gamma + \sqrt{\frac{\gamma}{2 \rho}}$, $\eta = \sqrt{\frac{\gamma}{2\rho}}$ and $\eta' = \frac{\mu \gamma}{2} - \sqrt{\frac{\gamma}{2\rho}}$ ensures
\begin{equation*}
     \E{e^{\frac{\mu \gamma}{2}T_k}(f(x_{T_k})-f^\ast)} \le f(x_0)-f^\ast.
\end{equation*}
\end{theorem}
\begin{proof}
To prove Proposition~\ref{prop:pl_cont}, we established that for $\xz$ solution of (\ref{eq:nest_continuized_sto}) with $\gamma \le \frac{1}{\rho L}$, $\gamma' = \gamma + \sqrt{\frac{\gamma}{2 \rho}}$, $\eta = \sqrt{\frac{\gamma}{2\rho}}$ and $\eta' = \frac{\mu \gamma}{2} - \sqrt{\frac{\gamma}{2\rho}}$, we have \eqref{eq:pl_cont}, namely 
     \begin{equation*}
   e^{\frac{\mu\gamma}{2}t}(f(x_t)-f^\ast) \le f(x_0)-f^\ast + M_t.
\end{equation*}

Applying Theorem~\ref{thm:martingal_stopping} with $\varphi_t :=  e^{\frac{\mu\gamma}{2}t}(f(x_t)-f^\ast)$, $K_0 := f(x_0)-f^\ast $ and with the almost surely finite stopping time $T_k$ yields
\[\E{ e^{\frac{\mu\gamma}{2}T_k}(f(x_{T_k})-f^\ast)}\le f(x_0)-f^\ast.  \]
We conclude using Proposition~\ref{prop:discretization_constant_param}, which states that if $\xz$ is generated with the stopping times $\tk$, then $\tilde x_k = x_{T_k}$, for $\xdisctilde$ generated by \eqref{alg:constant_param} with choice of parameters as stated.
\end{proof}
By definition of the random times $\{ T_k \}_{k\in \N^\ast}$, $T_k$ follows a Gamma law with shape parameter $k$ and rate $1$, such that $\E{T_k} = k$. In order to make comparisons with other algorithms, this motivates to interpret the bound of Theorem~\ref{thm:cv_pl} as a $\bigO\bpar{\exp(-\frac{\mu \gamma}{2}k)}$ bound in expectation. Convergence results trajectory-wise will however make the comparison more grounded, see Section~\ref{sec:high_proba}.
Even interpreting the rate of Theorem~\ref{thm:cv_pl} as $\bigO\bpar{\exp(-\frac{\mu \gamma}{2}k)} \approx O\bpar{(1-\frac{\mu \gamma}{2})^k}$ does not provide an improved convergence rate, compared to the existing results for gradient descent \citep{bolte2017error, karimi2016linear}, which are already optimal up to a constant factor \citep{PLlowerbound}.
In the deterministic case, Polyak's momentum was known to achieve a similar speed as gradient descent \citep{danilova2018non} when reducing momentum. However to our knowledge, there exists no convergence result of the Nesterov's momentum under the  assumptions of Theorem~\ref{thm:cv_pl}. 



\section{Improved Factor in Convergence Rate for Strongly quasar-convex Functions and Better Trajectory-Wise Guarantees}\label{sec:improv_gen}

We recall existing convergence guarantees for functions satisfying Assumptions~\hyperref[ass:l_smooth]{(L-smooth)} and \hyperref[ass:sqc]{(SQC$_{\tau,\mu}$)}, and use the notation $\kappa := \mu/L$. Gradient descent with stepsize $1/L$ satisfies $f(x_k)-f^\ast = \mathcal{O}((1-\tau\kappa)^k)$ \citep[Prop.~3]{hermant2024study}. In the same setting, the classical Nesterov scheme~\eqref{alg:nest_classic}, with a subroutine computing $\alpha_k$, achieves $f(x_k)-f^\ast = \mathcal{O}((1-\tau\sqrt{\kappa/2})^k)$ \citep{hinder2020near}, an improvement since typically $\kappa \ll 1$, at the cost of an additional logarithmic factor. Removing this subroutine, \citep{hermant2024study} recovers essentially the same rate $\mathcal{O}((1-\tau\sqrt{\kappa})^k)$, but requires an extra assumption controlling the lower curvature, namely $f(x) + \langle \nabla f(x), y-x\rangle + \tfrac{a}{2}\|x-y\|^2 \le f(y)$, for $a \le 0$.
This restricts how non-convex the function can be. Interestingly, this additional assumption is unnecessary in the continuous-time analysis \citep{hermant2024study}, revealing a gap between continuous and discrete guarantees, likely due to discretization effects. By nature, the continuized framework is less sensitive to discretization issues.

\begin{theorem}[{\citep[Theorem 2]{wang2023continuizedaccelerationquasarconvex}}]\label{thm:wang_et_al}
           Assume $T_1,\dots,T_k$ are random variables such that $T_{i+1} - T_i$ are  i.i.d  of law $\mathcal{E}(1)$, with convention $T_0 = 0$. Under Assumptions~\hyperref[ass:l_smooth]{(L-smooth)} and \hyperref[ass:sqc]{(SQC$_{\tau,\mu}$)}, iterations of \eqref{alg:constant_param} with $\eta' = \tau\sqrt{\frac{\mu}{L}}$, $\eta = \sqrt{\frac{\mu}{L}}$, $\gamma' =\sqrt{\frac{1}{\mu L}}$ and $\gamma = \frac{1}{ L}$ ensures
    \begin{equation*}            \mathbb{E}\left[e^{\tau\sqrt{\frac{\mu}{L}}T_k}(f(\tilde{x}_k)-f^\ast)\right] \le (f(x_0)-f^\ast + \mu\norm{x_0-x^\ast}^2).
         \end{equation*}
\end{theorem}
Up to identifying $T_k$ to $k$ (recall $\E{T_k}=k$), the above result recovers the rate $ \bigO( (1-\tau\sqrt{\kappa})^k)$, as it is approximately $\bigO\bpar{\exp(-\tau \sqrt{\kappa}k)}$ for $\tau \sqrt{\kappa}$ small.
Importantly, it does not assume extra lower-curvature control. This is a first example of a potential benefit of the continuized framework. In the subsequent sections, we show it has still more to offer.
\subsection{Improved Factor in Convergence Rate}\label{sec:sqc}
Under Assumption~\hyperref[ass:sqc]{(SQC$_{\tau,\mu}$)} with $\tau =1$, \citep{adr_qsc} show that the solution of the Heavy Ball equation \eqref{eq:HB} is such that $f(x_t)-f^\ast = \bigO(e^{-\sqrt{2\mu} t})$. This result was improving over the previously known bound $f(x_t)-f^\ast = \bigO(e^{-\sqrt{\mu} t})$ \citep{siegel2021accelerated}. This improvement of the convergence rate by a factor $\sqrt{2}$ was due to a refined choice of Lyapunov function.
To transfer this better rate in the discrete setting, authors however also assume convexity.  Inspired from their Lyapunov function, we consider the following one, adapted to the continuized framework
\begin{equation}\label{eq:main_text:lyapunov}
    \varphi(t,x_t,z_t)  := a_t(f(x_t)-f^\ast) + \frac{b_t}{2}\norm{z_t-x^\ast}^2 + \frac{c_t}{2}\norm{x_t-x^\ast}^2.
\end{equation}
The $\sqrt{2}$ factor improvement is due to taking $c_t < 0$, instead of $c_t = 0$, $\forall t \in \R_+$.  
We use this Lyapunov function to derive a new result in the discrete setting without assuming convexity. We begin with the continuous Lyapunov analysis, using Proposition~\ref{prop:sto_calc_abridged}.

\begin{theorem}\label{thm:sqc_cv_cont}
        Under Assumption~\hyperref[ass:l_smooth]{(L-smooth)}-\hyperref[ass:sqc]{(SQC$_{\tau,\mu}$)}-\hyperref[ass:sgc]{(SGC$_\rho$)}, let $(x_t,z_t)$ follows \eqref{eq:nest_continuized_sto} with $\gamma = \frac{1}{\rho L}$, $\gamma' = \frac{1}{\rho}\sqrt{\frac{2-\tau}{2\mu L}}$, $\eta = \frac{1}{\rho}\sqrt{\frac{2}{2-\tau}}\sqrt{\frac{\mu}{L}}$ and $\eta' = \frac{(1-\varepsilon) \tau}{\rho}\sqrt{\frac{1}{2(2-\tau)}}\sqrt{\frac{\mu}{L}}$   where $\varepsilon \in (0,1)$. Then as long as $\varepsilon \ge \bpar{\frac{\mu}{L}}^{\frac{1}{4}}$, we have
          \begin{equation}
             \mathbb{E}\left[ f(x_t)-f^\ast\right] \le \frac{1}{\varepsilon} e^{-\frac{(1-\varepsilon) \tau}{\rho}\sqrt{\frac{2}{(2-\tau)}}\sqrt{\frac{\mu}{L}}t}(f(x_0)-f^\ast + \mu\norm{x_0-x^\ast}^2).
         \end{equation}
     \end{theorem}Detailed proof of Theorem~\ref{thm:sqc_cv_cont} is deferred in Section~\ref{app:sqc}, in which we prove in particular
       \begin{equation}\label{eq:sqc_main_text_interm}
  e^{\frac{(1-\varepsilon) \tau}{\rho}\sqrt{\frac{2}{(2-\tau)}}\sqrt{\frac{\mu}{L}}t}(f(x_t)-f^\ast) \le \frac{1}{\varepsilon}\bpar{f(x_0)-f^\ast + \mu\norm{x_0-x^\ast}^2 + M_t},
         \end{equation}
    for some martingale $\mart$ such that $\E{M_t}=0$, $\forall t \in \R_+$. Evaluating \eqref{eq:sqc_main_text_interm} at $t = T_k$ and taking expectation, by Theorem~\ref{thm:martingal_stopping} and Proposition~\ref{prop:discretization_constant_param} we directly deduce the following result.
     \begin{theorem}\label{corr:sqc_cv_disc}
           Assume $T_1,\dots,T_k$ are random variables such that $T_{i+1} - T_i$ are  i.i.d  of law $\mathcal{E}(1)$, with convention $T_0 = 0$. Under Assumption~\hyperref[ass:l_smooth]{(L-smooth)}-\hyperref[ass:sqc]{(SQC$_{\tau,\mu}$)}-\hyperref[ass:sgc]{(SGC$_\rho$)}, iterations of (\ref{alg:constant_param}) with $\eta' = \frac{(1-\varepsilon) \tau}{\rho}\sqrt{\frac{1}{2(2-\tau)}}\sqrt{\frac{\mu}{L}}$, $\eta = \frac{1}{\rho}\sqrt{\frac{2}{2-\tau}}\sqrt{\frac{\mu}{L}}$, $\gamma' = \frac{1}{\rho}\sqrt{\frac{2-\tau}{2\mu L}}$ and $\gamma = \frac{1}{\rho L}$  where $\varepsilon \in (0,1)$ such that $\varepsilon \ge \bpar{\frac{\mu}{L}}^{\frac{1}{4}}$ ensures
    \begin{equation}            \mathbb{E}\left[e^{\frac{(1-\varepsilon) \tau}{\rho}\sqrt{\frac{2}{(2-\tau)}}\sqrt{\frac{\mu}{L}}T_k}(f(\tilde{x}_k)-f^\ast)\right] \le \frac{1}{\varepsilon}(f(x_0)-f^\ast + \mu\norm{x_0-x^\ast}^2).
         \end{equation}
     \end{theorem}
 The condition $\varepsilon \ge \bpar{\frac{\mu}{L}}^{\frac{1}{4}}$ is not an important restriction for many practical applications, as $\frac{\mu}{L}$ is typically very small. Note that a small $\varepsilon$ improves the rate in the exponential, leading to faster asymptotic convergence, but the $\varepsilon^{-1}$ factor makes the bound less interesting for a small number of iterations.
        
        \noindent
         \textbf{Comments \:}
     As already mentioned, the continuized framework has been applied to strongly quasar-convex functions in the specific case of deterministic gradients, \textit{i.e.} $\rho = 1$ in Assumption~\hyperref[ass:sgc]{(SGC$_\rho$)}.
         The result of Theorem~\ref{thm:sqc_cv_cont} is to compare with \citep{wang2023continuizedaccelerationquasarconvex} $$\mathbb{E}[ f( x_t ) - f(x^\ast) ]  \leq 
\left(  f(x_0) - f(x^\ast)  + \frac{\mu}{2} \| x_0 - x^\ast \|^2
\right) \exp\left( - \tau \sqrt{ \frac{\mu}{L} } t \right),$$ while Theorem~\ref{corr:sqc_cv_disc} is to compare with $$\mathbb{E}\left[ \exp\left( \tau \sqrt{ \frac{\mu}{L} } T_k \right) \left( f( \tilde{x}_k ) - f(x^\ast) \right) \right]  \leq 
 f(\tilde{x}_0) - f(x^\ast)  + \frac{\mu}{2} \| \tilde{x}_0 - x^\ast \|^2,
$$ see Theorem~\ref{thm:wang_et_al}. In the case $\tau = 1$, we improve the convergence rate by a factor $\sqrt{2}(1-\varepsilon)$. The improvement decreases as $\tau$ goes to zero. More than this improved factor in the convergence rate, it is worth noting that it is not clear if this improvement can be obtained using more classical methods than the continuized method.

\subsection{Convergence Results with High Probability and Almost Sure}\label{sec:high_proba}
Theorem~\ref{corr:sqc_cv_disc} states a result that holds in expectation. In applications, algorithms are typically executed only once, hence the importance of trajectory-wise convergence guarantees.

\noindent
\textbf{Convergence with High Probability \:}
We focus on finite time bounds, holding with some probability. We first recall the existing result of this type. 
\begin{corollary}[{\citep[Corollary 2]{wang2023continuizedaccelerationquasarconvex}}]\label{corr:wang_et_al}
  Under the same setting as Theorem~\ref{thm:wang_et_al}, with probability $1-\frac{1}{c_0} - \frac{1}{(1-c_1)^2k}$, for some $c_0 > 1$, $c_1 \in (0,1)$ we have
    \begin{equation*}          f(\tilde{x}_k)-f^\ast \le   c_0e^{-(1-c_1)\tau\sqrt{\frac{\mu}{L}}k}(f(x_0)-f^\ast + \mu\norm{x_0-x^\ast}^2).
         \end{equation*}
\end{corollary}
The above result ensures a convergence rate on $\{f(\tilde x_k)-f^\ast \}_{k\in \N}$ with probability $1-\frac{1}{c_0} - \frac{1}{(1-c_1)^2k}$. The convergence rate is close to the one of Theorem~\ref{thm:wang_et_al}, where a constant $c_0$ acts as a degrading factor, and $c_1$ degrades the convergence rate. The proof is based on the Markov inequality and the Chebyshev inequality. The latter is used in the proof to ensure $ \P(T_k \le c_1k)\le\frac{1}{(1-c_1)^2 k}$, $c_1\in (0,1)$. Importantly, the Chebyshev inequality is very general. As $T_k \sim \Gamma(k,1)$, it is in our case highly suboptimal. 
\begin{lemma}[Chernov inequality]\label{prop:chernov}\notag
    Let $T_k \sim \Gamma(k,1)$, $0<c_1 \le 1$. Then
    \begin{equation}\label{eq:chernov_proba}
        \P(T_k \le c_1k)\le e^{-(c_1-1-\log(c_1))k}.
    \end{equation}
\end{lemma}
Lemma~\ref{prop:chernov} bounds the probability that $T_k / k$ is below $c_1 \le 1$, which decreases exponentially fast for $c_1 < 1$. Intuitively, it allows to replace $T_k$ by $k$ in our convergence bounds, with high probability. Using this Lemma, we deduce a tighter result than Corollary~\ref{corr:wang_et_al}, that holds for general exponentially fast convergence rates holding in expectation.

\begin{proposition}\label{prop:conv_high_proba}
     Assume $T_k \sim \Gamma(k,1)$, $k \in \N^\ast$. If we have constants $\beta,~ K_0 > 0$ and a non-negative random variable $Y$ that verifies
                     \begin{equation}\label{eq:a.s.:2}   
                    \mathbb{E}\left[e^{\beta T_k}Y\right] \le K_0,
         \end{equation}
     then, with probability  $1-\frac{1}{c_0} - e^{-(c_1-1-\log(c_1))k}$, for some $c_0 > 1$ and $c_1 \in (0,1)$ we have
         \begin{equation*}
    Y \le c_0K_0e^{-\beta c_1 k}.
    \end{equation*}
\end{proposition}
See the proof of Proposition~\ref{prop:conv_high_proba} in Appendix~\ref{app:sqc_high_proba}. 
A direct application of Proposition~\ref{prop:conv_high_proba} gives the following result.
\begin{corollary}\label{cor:sqc_high_proba_1}
     Under the same setting as Theorem~\ref{corr:sqc_cv_disc}, for some $k \in \N^\ast$, with probability  $1-\frac{1}{c_0} - e^{-(c_1-1-\log(c_1))k}$, for some $c_0 > 1$, $c_1 \in (0,1)$, we have
        \begin{equation*}
         f(\tilde{x}_k) - f^\ast \le \frac{c_0}{\varepsilon}e^{-\frac{(1-\varepsilon) \tau}{\rho}\sqrt{\frac{2}{(2-\tau)}}\sqrt{\frac{\mu}{L}}c_1T_k}(f(x_0)-f^\ast + \mu\norm{x_0-x^\ast}^2).
    \end{equation*}
\end{corollary}
 The constant $c_0$ is a degrading factor, while $c_1$ degrades the convergence rates. The closer to $1$ are $c_0$ and $c_1$, the closer the bounds are to those that hold in expectation, at the cost of the result holding with lower probability. 
    Using the Chernov inequality instead of the Chebyshev inequality allows to significantly refines the results of \citep{wang2023continuizedaccelerationquasarconvex}, whose results hold with probability $1-\frac{1}{c_0} - \frac{1}{(1-c_1)^2k}$.  One can check that $1-\frac{1}{c_0} - e^{-(c_1-1-\log(c_1))k} \ge 1-\frac{1}{c_0} - \frac{1}{(1-c_1)^2k}$ for any $k \in \N$. Apart from the clear asymptotic improvement, it also appears in finite time, as for $c_1$ close to $1$, $\frac{1}{(1-c_1)^2k}$ is larger than $1$ for a significant amount of iterations, making the bound trivial. On the contrary, we have $e^{-(c_1-1-\log(c_1))k} \le 1$, $\forall k \in \N$, $c_1 \in (0,1)$. 

\paragraph{Almost Sure Convergence}
Focusing on asymptotic convergence results, we can obtain an almost sure convergence rate.  
\begin{proposition}\label{prop:almost_sure}
     Assume $T_1,\dots,T_k$ are random variables such that $T_{i+1} - T_i$ are  i.i.d  of law $\mathcal{E}(1)$, with convention $T_0 = 0$. If we have constants $\beta,~ K_0 > 0$ and a sequence of non-negative random variables $\{ Y_k\}_{k \in \N^\ast}$ that verify
                     \begin{equation*}       \forall k \in \N^\ast,~ \mathbb{E}\left[e^{\beta T_k}Y_k\right] \le K_0,
         \end{equation*}
          then for any $c\in (0,1)$ we have
         \begin{equation*}
   Y_k \overset{a.s.}{=} o\bpar{e^{-\beta c k}}.
\end{equation*}
\end{proposition}
See the proof of Proposition~\ref{prop:almost_sure} in Appendix~\ref{app:sqc_almostsure}. It
ensures we can transfer any exponentially fast results that holds in expectation to a result that holds almost surely, with a rate arbitrarily close to $e^{-\beta k}$. The convergence rate stated as $o(\cdot)$ instead of $\bigO(\cdot)$ typically occurs when studying asymptotic rates \citep{attouch2016rate,sebbouh2021almost,hermantgradient}. The core tool when proving Corollary~\ref{corr:as_sqc} is the Borel-Cantelli Lemma, which can be applied because the Chernov inequality (Lemma~\ref{prop:chernov}) ensures $\sum_{k\ge0} \P(T_k \le c_1k) < + \infty$.
In particular, the bound $\P(T_k \le c_1k)\le\frac{1}{(1-c_1)^2 k}$, $c_1\in (0,1)$ used in \citep{wang2023continuizedaccelerationquasarconvex} would be insufficient to deduce our result.
The following result is a direct application of Proposition~\ref{prop:almost_sure}.


\begin{corollary}\label{corr:as_sqc}
Under the same setting as Theorem~\ref{corr:sqc_cv_disc}, for any $c \in (0,1)$, one has
  \begin{equation}
     f(\tilde{x}_k)-f^\ast \overset{a.s.}{=} o \bpar{e^{-\frac{(1-\varepsilon) \tau}{\rho}\sqrt{\frac{2}{(2-\tau)}}\sqrt{\frac{\mu}{L}} c k} }.
 \end{equation}
\end{corollary}

\section{Non-Smooth Lyapunov Functions for Relaxing the Uniqueness Assumption}\label{sec:proj}
Because of their definition, strongly quasar-convex functions have a unique minimizer. A way to relax this limitation is to use the orthogonal projection operator onto $\minset$. It has been successfully used in the case of convex functions \citep{hyppo,aujol2024strong}, achieving near-optimal bounds for Nesterov momentum in a setting where former works assumed uniqueness of the minimizer.
\begin{bis}[SQC'$_{\mu}$]\label{ass:sqc_proj}
  f verifies
    $f^\ast \ge f(x) + \dotprod{\nabla f(x),\pi(x)-x} + \frac{\mu}{2}\norm{x-\pi(x)}^2$ for all $x\in \mathbb{R}^d$, some $\mu\in  \mathbb{R}^\ast_+$, and with $\pi(\cdot)$ being the orthogonal projection onto $\minset$.
\end{bis}
Assumption~\hyperref[ass:sgc]{(SQC'$_{\mu}$)} is a variant of Assumption~\hyperref[ass:sqc]{(SQC$_{\tau,\mu}$)} in the case $\tau = 1$, where an explicit minimizer $x^\ast$ is replaced by $\pi(x)$, the orthogonal projection of $x$ onto $\minset$. 
 This assumption is used in \citep{guptanesterov} under the name \textit{$\mu$-strong aiming condition}. If they can obtain a convergence result for the Heavy Ball dynamical system \eqref{eq:HB}, further obtaining a convergence result on the Nesterov momentum algorithm \eqref{alg:nest_classic} requires that the projection operator is affine-linear. This forces the set of minimizers to be affine, and indicates a discrepancy between continuous and discrete results. We further relax this assumption.
\begin{assumption}\label{ass:smooth_min}
    $\minset$ is a convex set with $C^2$ boundary, \textit{i.e.} its boundary is locally a $C^2$ sub-manifold.
\end{assumption}
The $C^2$ boundary assumption can be generalized to second order regularity \citep{shapiro2016differentiability, bonnans1998sensitivity}, which we leave aside for clarity of exposition.
Considering non-convex objects, ensuring $\pi(\cdot)$ is well defined is technically challenging. A feasible approach is to consider local definitions, namely defining $\pi(\cdot)$ for elements that are close enough to minimizers \citep{poliquin2000local,guptanesterov}.
In order to define it globally, it is however needed that $\minset$ is closed and convex \citep{jessen1940saetninger,busemann1947note,phelps1957convex}, such that it is the most general assumption we can make to ensure it is globally defined.

\subsection{A Technical Challenge: $\pi(\cdot)$ is not $C^1$}
A major building block of the convergence analysis using the continuized framework is Proposition~\ref{prop:sto_calc_abridged}, which assumes the Lyapunov function belongs to $C^1$. Under Assumption~\hyperref[ass:sgc]{(SQC'$_{\mu}$)}, we will construct our Lyapunov function so that it involves the projection $\pi(\cdot)$.
However, $\mathcal{X}^\ast$ being a closed convex set with $C^2$ boundary does not ensure $\pi$ to be $C^1$, as it is not differentiable on $\R^d~\backslash~ \partial \minset$ \citep{holmes1973smoothness,fitzpatrick1982differentiability}, where $\partial \minset$ is the boundary of $ \minset$. Proposition~\ref{prop:sto_calc_abridged} thus cannot be applied straightforwardly. We generalize it to include a broader set of Lyapunov functions.
\begin{proposition}[Non-Smooth Itô Formula]\label{prop:calc_sto_nondiff}
Let $x_t \in \R^d$ be a solution of 
\begin{equation*}
    dx_t = \zeta(x_t)dt + \int_\Xi G(x_{t^-},\xi)N(dt,d\xi),
\end{equation*}
where $\zeta : \R^d \to \R^d$ is a locally Lipschitz function and $G : \R^d \times \Xi \to \R^d$ is measurable. Let $\varphi:\R^d \to \R$ be locally Lipschitz. Then, denoting $\phi(t) = \varphi(x_t)$, there exists $\AA \subset \R_+$ of zero measure such that the derivative of $\phi$, denoted by $\dot{\phi}$, exists on $\R_+ \backslash \AA$.
Moreover, if 
$$\forall s \in [0,t] \backslash \AA,\quad \dot{\phi}(s) + \mathbb{E}_{\xi}[(\varphi(x_{s} + G(x_{s},\xi))-\varphi(x_{s}))] \le 0,$$
then we have
$$ \phi(t) \le \phi(0) + M_t,$$
where $\mart$ is a martingale such that $\E{M_t} = 0$ for all $t\in \R_+$.
\end{proposition}
\begin{proof}
Note first that $T_n \to +\infty$ almost surely. This can be seen as $T_n = \sum_{i=1}^n E_i$, $E_i \overset{i.i.d}{\sim} \mathcal{E}(1)$, and by the law of large numbers $T_n / n \to \E{1} = 1$, almost surely. This permits us to write
\begin{eqnarray}
    \begin{aligned}\label{eq:calc_sto_nondiff_1}
        \phi(t) - \phi(0) &= \sum_{k=0}^{+\infty} \phi(t \wedge T_{k+1}) -\phi(t \wedge T_{k}) \\
        &=\underbrace{\sum_{k=0}^{+\infty} \bpar{\phi(t \wedge T_{k+1}) -\phi(t \wedge T_{k+1}^-)}}_{\text{(I)}} + \underbrace{\sum_{k=0}^{+\infty}\bpar{\phi(t \wedge T_{k+1}^-) -\phi(t \wedge T_{k})}}_{\text{(II)}}. 
    \end{aligned}
\end{eqnarray}
We treat (I) and (II) separately. 

\noindent
\textbf{Sum (I) \quad} (I) describes the jumps. We have
\begin{eqnarray}
    \begin{aligned}\label{eq:calc_sto_nondiff_2}
\sum_{k=0}^{\infty}\phi(t\wedge T_{k+1})-\phi(t\wedge T_{k+1}^-)
&=\sum_{k\ge0:\,T_{k+1}\le t}\phi(T_{k+1})-\phi(T_{k+1}^-)\\
&=\sum_{k\ge0:\,T_{k+1}\le t}\varphi(x_{T_{k+1}^-} + G(x_{T_{k+1}^-}, \xi_{k+1}))-\varphi(x_{T_{k+1}^-}) \\
&=\int_0^t\int_\Xi\varphi(x_{s-}+G(x_{s-},\xi))-\varphi(x_{s-})\,N(ds,d\xi).
    \end{aligned}
\end{eqnarray}
\textbf{Sum (II) \quad}
The sum (II) is composed of $\phi(t \wedge T_{k+1}^-) -\phi(t \wedge T_{k})$ terms. For $s \in (t \wedge T_k, t \wedge T_{k+1})$, we have 
\begin{align}\label{eq:calc_sto_nondiff_x_diff}
\dot x_t =  \zeta(x_t),
\end{align}
such that $x_s$ is $C^1$ on $(t \wedge T_k, t \wedge T_{k+1})$.
Together with the fact that $\varphi$ is locally Lipschitz, it induces that $\phi(t) = \varphi(x_t)$ is locally Lipschitz on $(t \wedge T_k, t \wedge T_{k+1})$. Then, $\phi$ is locally absolutely continuous on $(t \wedge T_k, t \wedge T_{k+1})$. This means that its derivative $\dot \phi$ exists almost everywhere on $(t \wedge T_k, t \wedge T_{k+1})$, and that we have
$$\phi(t \wedge T_{k+1}^-) - \phi(t \wedge T_k) = \int_{t \wedge T_k}^{t \wedge T_{k+1}} \dot \phi(s)ds,  $$
where we used that $\phi$ is right-continuous at $T_k$ because $\xcont$ is càdlàg. Applying the same reasoning on each intervals $(t \wedge T_k, t \wedge T_{k+1})$, $k \in \N$, we have that the derivative of $\phi$ exists almost everywhere, and that Sum (II) rewrites

\begin{equation}\label{eq:calc_sto_nondiff_3:lebesgue}
    \sum_{k=0}^{+\infty}\bpar{\phi_{t \wedge T_{k+1}^-} -\phi_{t \wedge T_{k}}} = \sum_{k=0}^{+\infty}\int_{t \wedge T_{k}}^{t \wedge T_{k+1}} \dot{\phi}(s)ds = \int_0^t \dot{\phi}(s)ds
\end{equation}

\noindent
\textbf{Conclusion \:}
Combining \eqref{eq:calc_sto_nondiff_2} and \eqref{eq:calc_sto_nondiff_3:lebesgue}, we obtain
\begin{align*}
                   \phi(t)  &= \phi(0)+ \int_{0}^{t} \dot{\phi}(s)ds + \int_{[0,t]} (\varphi(x_{s^-} + G(x_{s^-},\xi))-\varphi(x_{s^-}))N(ds,d\xi)\\
         &=\phi(0)+ \int_{0}^{t} \dot{\phi}(s)ds + \int_{[0,t]} \mathbb{E}_{\xi}[(\varphi(x_{s} + G(x_{s},\xi))-\varphi(x_{s}))]ds\\
         &+ \int_{[0,t]}\int_\Xi (\varphi(x_{s^-} + G(x_{s^-},\xi))-\varphi(x_{s^-}))(N(ds,d\xi) -ds\mathcal{P}(d\xi))\\
         &=\phi(0)+ \int_{0}^{t} \dot{\phi}(s) +\mathbb{E}_{\xi}[(\varphi(x_{s} + G(x_{s},\xi))-\varphi(x_{s}))]ds + M_t,
\end{align*}
where $\mart$ is a martingale such that $\E{M_t} = 0$ for all $t\in \R_+$. Noting $\AA$ the set of null measure on which $\dot{\phi}$ may not be defined, by non-negativity of the Lebesgue integral, we have that if  
$$\forall s \in [0,t] \backslash \AA,\quad\dot{\phi}(s) + \mathbb{E}_{\xi}[(\varphi(x_{s} + G(x_{s},\xi))-\varphi(x_{s}))] \le 0,$$
then
$$ \phi(t) \le \phi(0) + M_t.$$

\end{proof}
Compared to Proposition~\ref{prop:sto_calc_abridged}, 
Proposition~\ref{prop:calc_sto_nondiff} provides a Lyapunov analysis framework in the case of non-$C^1$ Lyapunov functions. Precisely, it relaxes the $C^1$ assumption to allow locally Lipschitz functions. Intuitively, it indicates that even with non-differentiable Lyapunov function, we still have a rule to ensure $\E{\phi(t)} \le \phi(0)$. In this case, having to show $\dot{\phi}(s) + \mathbb{E}_{\xi}[(\varphi(x_{s} + G(x_{s},\xi))-\varphi(x_{s}))] \le 0$ plays the same role as establishing $ \dotprod{\nabla \varphi(x_s),\zeta(x_s)}  +    \mathbb{E}_{\xi} \varphi(x_s+ G(x_s)) - \varphi(x_s) \le 0 $ in the differentiable case, see \textit{e.g.} the proof of Proposition~\ref{prop:pl_cont}.

\begin{remark}
    As the function $f$ we minimize is typically involved in Lyapunov functions, Proposition~\ref{prop:calc_sto_nondiff} imposes $f$ to be locally Lipschitz. This is a natural assumption in non-smooth optimization, as it ensures that the slopes are locally bounded. However, it implies $f$ is continuous. It might be necessary to carry the continuized analysis, as we want to write $\phi$ as an integral in the proof of Proposition~\ref{prop:calc_sto_nondiff}.
\end{remark}

\subsection{Convergence Result}\label{sec:cv_proj}
 We consider the following Lyapunov function
\begin{equation}\label{eq:main_text:lyapunov_proj}
     \varphi(t,x_t,z_t)  := a_t(f(x_t)-f^\ast) + \frac{b_t}{2}\norm{z_t-\pi(x_t)}^2.
\end{equation}
This Lyapunov function is not $C^1$ because of the projection, and we cannot apply Proposition~\ref{prop:sto_calc_abridged}. However $\pi(\cdot)$ is Lipschitz, so in particular locally Lipschitz, such that Proposition~\ref{prop:calc_sto_nondiff} applies. 
\begin{theorem}\label{thm:sqc_proj}
Assume $T_1,\dots,T_k$ are random variables such that $T_{i+1} - T_i$ are  i.i.d  of law $\mathcal{E}(1)$, with convention $T_0 = 0$. Under Assumptions~\hyperref[ass:l_smooth]{(L-smooth)}-\hyperref[ass:sgc]{(SQC'$_{\mu}$)}-\hyperref[ass:sgc]{(SGC$_\rho$)}-\ref{ass:smooth_min},  iterations of (\ref{alg:constant_param}) with $\gamma = \frac{1}{ \rho L}$, $\gamma' = \frac{1}{\rho \sqrt{\mu L}}$, $\eta = \frac{1-\frac{1 + \sqrt{1+4\rho}}{2}\bpar{\frac{\mu}{L}}^{1/4}}{\rho}\sqrt{\frac{\mu}{L}}$ and $\eta' = \frac{1}{\rho}\sqrt{\frac{\mu}{L}}$ ensures the following bound

    \begin{equation}
    \E{e^{\eta T_k}(f(\tilde x_k)-f^\ast)} \le f(x_0)-f^\ast +\mu\bpar{1-\frac{1 + \sqrt{1+4\rho}}{2}\bpar{\frac{\mu}{L}}^{1/4}}\norm{x_0-\pi(x_0)}^2.
\end{equation}
\end{theorem}
The proof of Theorem~\ref{thm:sqc_proj} is in Section~\ref{app:proof_sqc_proj}. Its bound is non trivial as long as $\frac{\mu}{L} < \bpar{(1 + \sqrt{1+4\rho})/2}^{-4}$. In the deterministic gradient case, \textit{i.e.} $\rho = 1$, this upper bound becomes $\bpar{ (1 + \sqrt{1+4\rho})/2}^{-4} = \bpar{ (1 + \sqrt{5})/2}^{-4} \approx 0.15$, which is a reasonable bound as typically in most applications we expect that $\frac{\mu}{L} \ll  1$. Theorem~\ref{thm:sqc_proj} ensures an accelerated convergence rate relaxing the assumptions holding in \citep{wang2023continuizedaccelerationquasarconvex} (unique minimizer, deterministic gradient) and \citep{guptanesterov} (lower curvature restriction, $\minset$ is affine-linear).
\begin{remark}
    Note that the Lyapunov function \eqref{eq:main_text:lyapunov_proj} used to prove Theorem~\ref{thm:sqc_proj} does not incorporate the negative term used in the Lyapunov function \eqref{eq:main_text:lyapunov} to prove Theorem~\ref{corr:sqc_cv_disc}, in the case of uniqueness of the  minimizer. Investigating whether this modification could yield a slightly better rate is a potential direction for future work.
\end{remark}
\begin{remark}
    The assumption that $\minset$ is convex helps to ensure $\pi(\cdot)$ is a well defined 
    function.
    It is however important to note that in the computation, this assumption is only used to have that the projection is Lipschitz with constant $1$. It indicates that if the trajectory is such that along its path, $\pi(\cdot)$ is well defined with existing partial derivatives, then this assumption could be relaxed by considering a higher Lipschitz constant for the projection. 
\end{remark}

\paragraph{Result Trajectory-Wise with High Probability and Almost Surely.}
Following the discussion of Section~\ref{sec:high_proba}, we extend the result of Theorem~\ref{thm:sqc_proj} to results that hold trajectory-wise.

\begin{corollary}\label{cor:sqc_high_proba_2}
   Under the same setting as for Theorem~\ref{thm:sqc_proj}, one has that:
   \begin{enumerate}
       \item     With probability  $1-\frac{1}{c_0} - e^{-(c_1-1-\log(c_1))k}$, $c_0 > 1$, $c_1 \in (0,1)$,    \begin{equation*}
         f(\tilde{x}_k) - f^\ast \le c_0\bpar{ f(x_0)-f^\ast) +\mu\bpar{1-\frac{1 + \sqrt{1+4\rho}}{2}\bpar{\frac{\mu}{L}}^{1/4}}\norm{x_0-\pi(x_0)}^2}e^{-\eta c_1 k}.
    \end{equation*}
    \item For any $c \in (0,1)$,
  \begin{equation*}
     f(\tilde{x}_k)-f^\ast \overset{a.s.}{=} o \bpar{e^{-\eta c  k } }.
 \end{equation*}
   \end{enumerate}

\end{corollary}
Corollary~\ref{cor:sqc_high_proba_2} is a direct application of Proposition~\ref{prop:conv_high_proba} and Proposition~\ref{prop:almost_sure}.

\section{Convergence Proofs}
In this section we prove our main convergence results. 

\subsection{Proof of Theorem~\ref{thm:sqc_cv_cont} }\label{app:sqc}
We start by showing a sufficient decrease of the Lyapunov function.
\begin{lemma}\label{lem:sqc_lyapunov_decrease}
Let $(x_t,z_t)$ defined by \eqref{eq:nest_continuized_sto}. Let Assumptions~\hyperref[ass:l_smooth]{(L-smooth)}-\hyperref[ass:sqc]{(SQC$_{\tau,\mu}$)}-\hyperref[ass:sgc]{(SGC$_\rho$)} hold. We define the following Lyapunov function
\begin{equation}
    \phi(t) = a_t(f(x_t)-f^\ast) + \frac{b_t}{2}\norm{z_t - x^\ast}^2 + \frac{c_t}{2}\norm{x_t - x^\ast}^2.
\end{equation}
We fix $a_t = e^{2 \eta't}$, $b_t = \frac{2 \mu}{2-\tau} e^{2 \eta't}$, $c_t = - (1-\varepsilon) \mu \frac{\tau}{2-\tau} e^{2 \eta't}$, $\eta' = \frac{(1-\varepsilon) \tau}{\rho}\sqrt{\frac{1}{2(2-\tau)}}\sqrt{\frac{\mu}{L}}$, and $\eta = \frac{1}{\rho}\sqrt{\frac{2}{2-\tau}}\sqrt{\frac{\mu}{L}}$ , $\gamma' = \frac{1}{\rho}\sqrt{\frac{2-\tau}{2\mu L}}$ and $\gamma = \frac{1}{\rho L}$ where $\varepsilon \in (0,1)$. Under the assumption $\varepsilon \ge \bpar{\frac{\mu}{L}}^{\frac{1}{4}}$, we have
         \begin{equation}
         \phi(t) \le \phi(0) + M_t,
     \end{equation}
     where by definition
     \begin{equation}
          \phi(t) = e^{\frac{(1-\varepsilon) \tau}{\rho}\sqrt{\frac{2}{(2-\tau)}}\sqrt{\frac{\mu}{L}}t}\bpar{f(x_t)-f^\ast +  \frac{ \mu}{2-\tau} \norm{z_t - x^\ast}^2 + \frac{- (1-\varepsilon) \mu}{2} \frac{\tau}{2-\tau} \norm{x_t - x^\ast}^2}.
     \end{equation}
\end{lemma}

\begin{proof}
Let $\overline{x}_t= (t,x_t,z_t)$. It satisfies $d\overline{x}_t = \zeta(\overline{x}_t)dt + G(\overline{x}_t,\xi)N(dt,d\xi)$, where
\begin{equation}
    \zeta(\overline{x}_t) = \begin{pmatrix}
    1 \\ \eta(z_t-x_t) \\ \eta'(x_t-z_t)
    \end{pmatrix},\quad G(\overline{x}_t,\xi) = \begin{pmatrix}
    0 \\ -\gamma \nabla f(x_t,\xi) \\ -\gamma' \nabla f(x_t,\xi).
    \end{pmatrix}
\end{equation}
We apply Proposition~\ref{prop:sto_calc_abridged} to $\phi_t= \varphi(\overline{x}_t)$, where
\begin{equation}
    \varphi(t,x,z) =  a_t(f(x)-f^\ast) + \frac{b_t}{2}\norm{z - x^\ast}^2 + \frac{c_t}{2}\norm{x - x^\ast}^2.
\end{equation}
so 
\begin{equation}\label{eq:prop_calc_sto}
    \varphi(\overline{x}_t) = \varphi(\overline{x}_0) + \int_{0}^t \dotprod{\nabla \varphi(\overline{x}_s),\zeta(\overline{x}_s)}ds + \int_{[0,t]} \mathbb{E}_\xi( \varphi(\overline{x}_s + G(\overline{x}_s,\xi)) - \varphi(\overline{x}_s))ds + M_t,
\end{equation}
where $M_t$ is a martingale.
We have 
\begin{align}
    \frac{\partial \varphi}{\partial s} &= \frac{da_s}{ds}(f(x)-f^\ast) + \frac{1}{2}\frac{db_s}{ds}\norm{z - x^\ast}^2 + \frac{1}{2}\frac{dc_s}{ds}\norm{x - x^\ast}^2, \\
    \frac{\partial \varphi}{\partial x} &= a_s\nabla f(x) + c_s(x-x^\ast),\\
    \quad  \frac{\partial \varphi}{\partial z} &= b_s(z - x^\ast).
\end{align}
So, 
\begin{align}
    \dotprod{\nabla \varphi(\overline{x}_s),\zeta(\overline{x}_s)} &=\frac{da_s}{ds}(f(x_s)-f^\ast) + \frac{1}{2}\frac{db_s}{ds}\norm{z_s - x^\ast}^2 + \frac{1}{2}\frac{dc_s}{ds}\norm{x_s - x^\ast}^2 \nonumber \\
    &+\dotprod{a_s \nabla f(x_s) + c_s(x_s - x^\ast),\eta(z_s-x_s)} + \dotprod{b_s(z_s-x^\ast),\eta'(x_s - z_s)} \nonumber\\
    &=\frac{da_s}{ds}(f(x_s)-f^\ast) + \frac{1}{2}\frac{db_s}{ds}\norm{z_s - x^\ast}^2 + \frac{1}{2}\frac{dc_s}{ds}\norm{x_s - x^\ast}^2 \label{eq:sqc_eq_1}\\
    &+a_s \eta \dotprod{ \nabla f(x_s) ,z_s-x_s}\label{eq:sqc_eq_2}\\
    &+ c_s \eta\dotprod{x_s - x^\ast,z_s-x_s} + b_s \eta' \dotprod{z_s-x^\ast,x_s - z_s}\label{eq:sqc_eq_3}
\end{align}
Also, 
\begin{align}
   \mathbb{E}_{\xi}[ \varphi(\overline{x}_s+ G(\overline{x}_s)) - \varphi(\overline{x}_s) ]&= \mathbb{E}_{\xi}[a_s(f(x_s - \gamma \nabla f(x_s,\xi)) - f(x_s))] \nonumber\\
   &+ \frac{b_s}{2}\mathbb{E}_{\xi}[\norm{ z_s - \gamma'\nabla f(x_s,\xi) - x^\ast}^2 - \norm{z_s -x^\ast}^2]\nonumber\\
    &+\frac{c_s}{2}\mathbb{E}_{\xi}[\norm{ x_s - \gamma\nabla f(x_s,\xi) - x^\ast}^2 - \norm{x_s - x^\ast}^2]\nonumber\\
    &= a_s\mathbb{E}_{\xi}[f(x_s - \gamma \nabla f(x_s,\xi)) - f(x_s)]  \label{eq:sqc_eq_4}\\
    & + \frac{b_s \gamma^{'2} + c_s \gamma^2}{2}\mathbb{E}_{\xi}[\norm{\nabla f(x_s,\xi)}^2]\label{eq:sqc_eq_5}\\
    &+ b_s\gamma' \dotprod{\nabla f(x_s),x^\ast - z_s} + c_s \gamma \dotprod{\nabla f(x_s),x^\ast - x_s},\label{eq:sqc_eq_6}
\end{align}
where we used $\mathbb{E}_{\xi}[\nabla f(x_s,\xi)] = \nabla f(x_s)$.
Using Assumptions \hyperref[ass:l_smooth]{(L-smooth)}-\hyperref[ass:sgc]{(SGC$_\rho$)} + Lemma~\ref{lem:descent_sgc}, we have
\begin{align}
    (\ref{eq:sqc_eq_4}) + (\ref{eq:sqc_eq_5}) \le \bpar{\rho b_s \gamma^{'2} + \rho c_s \gamma^2 + a_s \gamma( \rho\gamma L-2)}\frac{1}{2}\norm{\nabla f(x_s)}^2.
\end{align}
Also,
\begin{align}\label{eq:scq_eq_7}
    (\ref{eq:sqc_eq_6}) = (b_s \gamma' + c_s \gamma)\dotprod{\nabla f(x_s),x^\ast - x_s} + b_s \gamma' \dotprod{\nabla f(x_s),x_s-z_s}.
\end{align}
Using $\tau,\mu$-strong quasar-convexity, we get
\begin{align}
    (\ref{eq:sqc_eq_2}) + (\ref{eq:scq_eq_7}) \le -(b_s \gamma' + c_s \gamma)\tau\bpar{f(x_s)-f^\ast + \frac{\mu}{2}\norm{x_s-x^\ast}} + (a_s \eta - b_s \gamma')\dotprod{\nabla f(x_s),z_s - x_s}.
\end{align}
Also,
\begin{align}
    (\ref{eq:sqc_eq_3}) = &-b_s \eta' \norm{z_s - x^\ast}^2 +  c_s \eta\dotprod{x_s - x^\ast,z_s-x_s} + b_s\eta' \dotprod{z_s - x^\ast,x_s - x^\ast}\\
    = &-b_s\eta' \norm{z_s - x^\ast}^2 + b_s \eta' \norm{x_s - x^\ast}^2,
\end{align}
where we assumed $c_s \eta = -b_s \eta'\Rightarrow c_s = -\frac{\eta'}{\eta}b_s$.
We now put everything together 
\begin{align}
     \dotprod{\nabla \varphi(x_s),\zeta(x_s)} +  \mathbb{E}_{\xi}[ \varphi(\overline{x}_s+ G(\overline{x}_s)) - \varphi(\overline{x}_s) ]&\le \bpar{\frac{da_s}{ds}-\tau(b_s \gamma'+ c_s \gamma)}(f(x_s)-f^\ast) \label{eq:sqc_constraint_1}\\
     &+\bpar{\frac{db_s}{ds} - 2b_s \eta'}\frac{1}{2}\norm{z_s-x^\ast}^2\label{eq:sqc_constraint_2} \\
     &+ \bpar{\frac{dc_s}{ds}
     -\mu \tau(b_s \gamma' + c_s \gamma) + 2b_s \eta'}\frac{1}{2}\norm{x_s-x^\ast}^2\label{eq:sqc_constraint_3}\\
     &+(a_s \eta - b_s \gamma')\dotprod{\nabla f(x_s),z_s - x_s}\label{eq:sqc_constraint_4}\\
     &+(\rho b_s \gamma^{'2} + \rho c_s \gamma^2 + a_s \gamma( \rho\gamma L-2))\frac{1}{2}\norm{\nabla f(x_s)}^2.\label{eq:sqc_constraint_5}
\end{align}

\paragraph{Parameter Choices}
Inspired from computations on deterministic damped system, we fix $a_s = e^{2 \eta's}$, $b_s = \frac{2 \mu}{2-\tau} e^{2 \eta's}$, $c_s = - \theta \mu \frac{\tau}{2-\tau} e^{2 \eta's}$, with $\theta > 0$ to be defined later. Also, $\eta' = \frac{\theta \tau}{\rho}\sqrt{\frac{1}{2(2-\tau)}}\sqrt{\frac{\mu}{L}}$, and $\eta = \frac{1}{\rho}\sqrt{\frac{2}{2-\tau}}\sqrt{\frac{\mu}{L}} \Rightarrow \frac{\eta'}{\eta} = \frac{\tau \theta}{2}$. Also, we fix $\gamma' = \frac{1}{\rho}\sqrt{\frac{2-\tau}{2\mu L}}$ and $\gamma = \frac{1}{\rho L}$.

The exponential rate ensures the term of \eqref{eq:sqc_constraint_2} is non-positive.

Also, the term of \eqref{eq:sqc_constraint_4} is zero as long as $a_s \eta - b_s \gamma' = 0 \Leftrightarrow \eta = \frac{2\mu}{2-\tau} \frac{1}{\rho}\sqrt{\frac{2-\tau}{2\mu L}} = \frac{1}{\rho}\sqrt{\frac{2}{2-\tau}}\sqrt{\frac{\mu}{L}}$.

The term of \eqref{eq:sqc_constraint_1} is
 \begin{align}
       a_s^{-1} e^{-2\eta' s}\bpar{\frac{da_s}{ds}-\tau(b_s \gamma'+ c_s \gamma)} &= 2\eta' - 2\tau \frac{\mu}{2-\tau}\gamma' + \theta \mu \frac{\tau^2}{2-\tau} \gamma \\
        &=2\frac{\theta \tau}{\rho}\sqrt{\frac{1}{2(2-\tau)}}\sqrt{\frac{\mu}{L}} - 2\frac{\mu \tau}{2-\tau}\frac{1}{\rho}\sqrt{\frac{2-\tau}{2\mu L}} +\theta \mu \frac{\tau^2}{2-\tau}\frac{1}{L\rho}\\
        &= \frac{\tau}{\rho}\sqrt{\frac{2}{2-\tau}}\sqrt{\frac{\mu}{L}}\left( \theta - 1 + \frac{\theta \tau}{\sqrt{2(2-\tau)}} \sqrt{\frac{\mu}{L}}\right).\label{eq:sqc_constraint_last}
\end{align}
 Omitting exponential factors, the term of \eqref{eq:sqc_constraint_5} is
    \begin{align}
        \rho b_s \gamma^{'2} + \rho c_s \gamma^2 + a_s \gamma( \rho\gamma L-2)) &= \rho \frac{2\mu}{2-\tau}\frac{1}{\rho^2}\bpar{\sqrt{\frac{2-\tau}{2\mu L}}}^2 -\rho \theta \mu \frac{\tau}{2-\tau}\frac{1}{\rho^2 L^2} - \frac{1}{\rho L} \nonumber\\
        &=-\theta \frac{\tau}{2-\tau}\frac{\mu}{\rho L^2}\label{eq:sqc_remaining_gradient}.
    \end{align}

      The term of \eqref{eq:sqc_constraint_3} is
    \begin{align}
        &\frac{dc_s}{ds}
     -\mu \tau (b_s \gamma' + c_s \gamma) + 2b_s \eta' = \nonumber\\
     &-2\frac{\theta \tau}{\rho}\sqrt{\frac{1}{2(2-\tau)}}\sqrt{\frac{\mu}{L}}\theta \mu \frac{\tau}{2-\tau} - \mu \tau  \frac{2 \mu}{2-\tau}\frac{1}{\rho}\sqrt{\frac{2-\tau}{2\mu L}} + \mu \tau \theta \mu \frac{\tau}{2-\tau}\frac{1}{\rho L} + 2\frac{2 \mu}{2-\tau} \frac{\theta \tau}{\rho}\sqrt{\frac{1}{2(2-\tau)}}\sqrt{\frac{\mu}{L}}\nonumber\\
     &=\mu\sqrt{\frac{\mu}{L}}\frac{\tau}{\rho}\bpar{-\theta^2\sqrt{\frac{2}{2-\tau}}\frac{\tau}{2-\tau} - \sqrt{\frac{2}{2-\tau}} + \frac{\tau \theta }{2-\tau}\sqrt{\frac{\mu}{L}} + 2\theta\sqrt{\frac{2}{2-\tau}}\frac{1}{2-\tau}}\nonumber\\
     &=\sqrt{\frac{2}{2-\tau}}\frac{1}{2-\tau}\mu\sqrt{\frac{\mu}{L}}\frac{\tau}{\rho}\bpar{-\tau \theta^2 -2 + \tau + \tau \theta\sqrt{\frac{2-\tau}{2}}\sqrt{\frac{\mu}{L}} + 2\theta}.\label{eq:finlization}
     \end{align}
      \eqref{eq:finlization} is not straightforward to get cancelled. Symbolic computations in Python show that the maximal solution is $$\theta_- = \frac{1}{\tau} + \frac{1}{4}\sqrt{2-\tau}\sqrt{\frac{2\mu}{L}}-\frac{1}{2\sqrt{2}\tau}\sqrt{4\sqrt{2} \tau\sqrt{2-\tau}\sqrt{\frac{\mu}{L}} -\tau^2(\tau-2)\frac{\mu}{L}+8(1-\tau)^2}, $$ but the expression is barely readable. Otherwise, getting rid of the factor $\sqrt{\frac{2}{2-\tau}}\frac{1}{2-\tau}\mu\sqrt{\frac{\mu}{L}}\frac{\tau}{\rho}$, fixing $\theta = 1-\alpha$,  \eqref{eq:finlization} becomes
\begin{align*}
     &-\tau(1-\alpha)^2 - 2 + \tau +  (1-\alpha) \tau \sqrt{\frac{2-\tau}{2}}\sqrt{\frac{\mu}{L}}  +2-2\alpha \\
     &\le -\tau \alpha^2 + 2(\tau-1)\alpha +\tau \sqrt{\frac{2-\tau}{2}}\sqrt{\frac{\mu}{L}}\\
     &\le -\tau \alpha^2 +\tau \sqrt{\frac{2-\tau}{2}}\sqrt{\frac{\mu}{L}}.
\end{align*}
       The last inequality uses $\tau-1 \le 1$. Now, fixing $\alpha \ge c \tau^\alpha \bpar{\frac{\mu}{L}}^{\beta}$, for some non-negative constants $c$, $\alpha$ and $\beta$, observe that $c = 1$, $\alpha = 0$ and $\beta = 1/4$ is sufficient to have   \eqref{eq:finlization} upper-bounded by 0. Plugging this choice in \eqref{eq:sqc_constraint_last}, we have that \eqref{eq:sqc_constraint_1} is non positive as
     \begin{align}
         \theta - 1 + \frac{\theta \tau}{\sqrt{2(2-\tau)}} \sqrt{\frac{\mu}{L}} &\le - \bpar{\frac{\mu}{L}}^{\frac{1}{4}}  + \frac{(1 - \bpar{\frac{\mu}{L}}^{\frac{1}{4}}) \tau}{\sqrt{2(2-\tau)}} \sqrt{\frac{\mu}{L}}\\
         &\le  - \bpar{\frac{\mu}{L}}^{\frac{1}{4}}  + \frac{\tau}{\sqrt{2(2-\tau)}} \sqrt{\frac{\mu}{L}}\\
         &=-\bpar{\frac{\mu}{L}}^{\frac{1}{4}} \bpar{1-\frac{\sqrt{\tau}}{2^{\frac{1}{4}}\sqrt{2-\tau}}\bpar{\frac{\mu}{L}}^{\frac{1}{4}}}\\
         &\le 0,
     \end{align}
     the last inequality holding as $\tau \le 1$ and $\frac{\mu}{L}\le 1$.
     
     

     In particular, fixing $\theta = 1-\varepsilon$ with $\varepsilon \in (0,1)$ such that $\varepsilon \ge \bpar{\frac{\mu}{L}}^{\frac{1}{4}}$, we have
     \begin{equation}
         \phi(t) \le \phi(0) + M_t,
     \end{equation}
     where 
     \begin{equation}\label{eq:sqc_decrease_lyapunov}
          \phi(t) = e^{\frac{(1-\varepsilon) \tau}{\rho}\sqrt{\frac{2}{(2-\tau)}}\sqrt{\frac{\mu}{L}}t}\bpar{f(x_t)-f^\ast +  \frac{ \mu}{2-\tau} \norm{z_t - x^\ast}^2 + \frac{- (1-\varepsilon) \mu}{2} \frac{\tau}{2-\tau} \norm{x_t - x^\ast}^2}.
     \end{equation}
     \end{proof}
      \textbf{Proof of Theorem~\ref{thm:sqc_cv_cont} \:.} We use the following result, ensuring a quadratic growth property.
         \begin{lemma}{\citep[Corollary 1]{hinder2020near}}\label{cor_sqc_implies_qg}
             If $f$ verifies Assumption~\hyperref[ass:sqc]{(SQC$_{\tau,\mu}$)}, we have
             \begin{equation*}
                 f(x)-f^\ast \ge \frac{\mu \tau}{2(2-\tau)}\norm{x-x^\ast}^2,\quad \forall x \in \mathbb{R}^d.
             \end{equation*}
         \end{lemma}
         Recalling \eqref{eq:sqc_decrease_lyapunov}, using Lemma~\ref{cor_sqc_implies_qg}, we have
         \begin{equation}
             \phi(t) \ge \varepsilon e^{\frac{(1-\varepsilon) \tau}{\rho}\sqrt{\frac{2}{(2-\tau)}}\sqrt{\frac{\mu}{L}}t}(f(x_t)-f^\ast).
         \end{equation}
         So, by Lemma~\ref{lem:sqc_lyapunov_decrease}, we have
         \begin{equation}\label{eq:sqc_functional_as_decrease}
             e^{\frac{(1-\varepsilon) \tau}{\rho}\sqrt{\frac{2}{(2-\tau)}}\sqrt{\frac{\mu}{L}}t}(f(x_t)-f^\ast) \le \frac{1}{\varepsilon}\bpar{\phi(0) + M_t}.
         \end{equation}
         Taking expectation, we have
         \begin{equation}
             \mathbb{E}\left[ f(x_t)-f^\ast\right] \le \frac{1}{\varepsilon} e^{-\frac{(1-\varepsilon) \tau}{\rho}\sqrt{\frac{2}{(2-\tau)}}\sqrt{\frac{\mu}{L}}t}\phi(0).
         \end{equation}
 Using $x_0 = z_0$, and $\tau \ge 1$, by the expression of $\phi(0)$ we have
 \begin{equation}
      \mathbb{E}\left[ f(x_t)-f^\ast\right] \le \frac{1}{\varepsilon} e^{-\frac{(1-\varepsilon) \tau}{\rho}\sqrt{\frac{2}{(2-\tau)}}\sqrt{\frac{\mu}{L}}t}(f(x_0)-f^\ast + \mu\norm{x_0-x^\ast}^2).
 \end{equation}

\subsection{Proof of Theorem~\ref{thm:sqc_proj}}\label{app:proof_sqc_proj}

Let $\overline{x}_t= (t,x_t,z_t)$. It satisfies $d\overline{x}_t = \zeta(\overline{x}_t)dt + G(\overline{x}_t,\xi)N(dt,d\xi)$, where
\begin{equation}
    \zeta(\overline{x}_t) = \begin{pmatrix}
    1 \\ \eta(z_t-x_t) \\ \eta'(x_t-z_t)
    \end{pmatrix},\quad G(\overline{x}_t,\xi) = \begin{pmatrix}
    0 \\ -\gamma \nabla f(x_t,\xi) \\ -\gamma' \nabla f(x_t,\xi).
    \end{pmatrix}
\end{equation}
We define
\begin{equation}
    \varphi(t,x,z) =  a_t(f(x)-f^\ast) + \frac{b_t}{2}\norm{\pi(x) - z}^2, 
\end{equation}
with $\pi(\cdot)$ being the orthogonal projection onto $\mathcal{X}^\ast :=\arg \min f$, namely 
\begin{equation}
    \forall x\in \R^d, ~ \pi(x) = \arg \min_{y\in\mathcal{X}^\ast} \norm{y-x}^2.
\end{equation}

The property of the orthogonal projection onto a convex set has been studied extensively \citep{ZARANTONELLO1971237}. $\mathcal{X}^\ast$ being a closed convex set with $C^2$ boundary ensures $\pi$ is well defined on $\R^d$ \citep{holmes1973smoothness,fitzpatrick1982differentiability}, and in particular, it is $1$-Lipschitz. As $f$ is $C^1$, we have that $\varphi$ is locally Lipschitz.
Noting $\phi(t) := \varphi(\overline{x}_t)$, we apply Proposition~\ref{prop:calc_sto_nondiff}, such that
$\dot{\phi}$ the derivative of $\phi$ exists on $\R_+ \backslash \mathbb{A}$, where $\mathbb{A}$ is a set of null measure. In particular on $\R_+ \backslash \mathbb{A}$, $\xz$ are $C^1$ and verify
\begin{align}\label{eq:smooth_equations}\left\{
    \begin{array}{ll}
        \dot{x}_s &= \eta(z_s-x_s)  \\
        \dot{z}_s &= \eta'(x_s-z_s)
    \end{array}
\right.
\end{align}
So, on $\R_+ \backslash \mathbb{A}$, one has
\begin{equation*}
    \phi(s) = \frac{da_s}{ds}(f(x_s)-f^\ast) + \frac{1}{2}\frac{db_s}{ds}\norm{\pi(x_s) - z_s}^2+ a_s\dotprod{\nabla f(x_s),\dot{x}_s} + b_s\dotprod{\pi(x_s)-z_s, \partial \pi(x_s,\dot x_s) - \dot{z}_s},
\end{equation*}
where $\partial \pi(x_s,\dot x_s) $ denotes the directional derivative of $\pi(\cdot)$ at point $x_s$ in the direction $\dot x_s$. Then, with some computations
\begin{eqnarray}
    \begin{aligned}\label{eq:non_diff_1}
    \dot{\phi}(s) &= \frac{da_s}{ds}(f(x_s)-f^\ast) + \frac{1}{2}\frac{db_s}{ds}\norm{\pi(x_s) - z_s}^2 
+ a_s\dotprod{\nabla f(x_s),\dot{x}_s} + b_s\dotprod{\pi(x_s)-z_s, \partial \pi(x_s,\dot x_s)  - \dot{z}_s}\\
&=\frac{da_s}{ds}(f(x_s)-f^\ast) + \frac{1}{2}\frac{db_s}{ds}\norm{\pi(x_s) - z_s}^2+ \eta a_s\dotprod{\nabla f(x_s),z_s-x_s} \\
&+ b_s\dotprod{\pi(x_s)-z_s, \partial \pi(x_s,\dot x_s)  - \eta'(x_s-z_s)}\\
&=\frac{da_s}{ds}(f(x_s)-f^\ast) + \frac{1}{2}\frac{db_s}{ds}\norm{\pi(x_s) - z_s}^2+ \eta a_s\dotprod{\nabla f(x_s),z_s-x_s} +\eta' b_s\dotprod{ z_s-\pi(x_s),x_s-z_s}\\
&+b_s\dotprod{\pi(x_s)-x_s,\partial \pi(x_s,\dot x_s) } + b_s \dotprod{x_s-z_s, \partial \pi(x_s,\dot x_s) }\\
&=\frac{da_s}{ds}(f(x_s)-f^\ast) + \frac{1}{2}\frac{db_s}{ds}\norm{\pi(x_s) - z_s}^2+ \eta a_s\dotprod{\nabla f(x_s),z_s-x_s} +\eta' b_s\dotprod{ z_s-\pi(x_s),x_s-z_s}\\
&+b_s\dotprod{\pi(x_s)-x_s,\partial \pi(x_s,\dot x_s) } - \frac{b_s}{\eta} \dotprod{\dot x_s,\partial \pi(x_s,\dot x_s)}.
    \end{aligned}
\end{eqnarray}
In the last equality, we use  $\dotprod{x_s-z_s, \partial \pi(x_s,\dot x_s) } = -\frac{1}{\eta}\dotprod{\dot x_s,\partial \pi(x_s,\dot x_s)}$, induced by \eqref{eq:smooth_equations}.
The two last terms of \eqref{eq:non_diff_1} would automatically be zero if $\minset$ was reduced to a singleton, as the projection would be identically equal to $x^\ast$. We show that in our case we can upper-bound these terms by zero.
\begin{lemma}\label{lem:orthogon_proj}
    Let $U$ a subset of $\R^d$ and $K \subset U$ such that the orthogonal projection $P : U \to K$ is well defined. Then, 
    \begin{enumerate}
        \item    
        $\forall x \in U,y\in K,\quad  \dotprod{x-y,y-P(x)} \le 0. $
   \item $\forall x,d \in U,\quad \dotprod{P(x+d)-P(x),d} \ge 0.$
    \end{enumerate}
\end{lemma}
These are classic properties, but are often proved under the underlying assumption that $K$ is convex. We provide a proof in Appendix~\ref{app:orthogon_proj}. Applying Lemma~\ref{lem:orthogon_proj}-1. with $P = \pi$ to $x = x_s+hd$, $d = \dot x_s$ and $y = \pi(x_s)$ gives
\begin{align*}
 \dotprod{x_s+hd-\pi(x_s),\pi(x_s)-\pi(x_s+hd)} \le 0.
\end{align*}
We split this scalar product between two terms
\begin{align*}
 \dotprod{x_s-\pi(x_s),\pi(x_s)-\pi(x_s+hd)} + h\dotprod{d,\pi(x_s)-\pi(x_s+hd)} \le 0.
\end{align*}
Dividing by $h>0$, we obtain
\begin{align*}
 \dotprod{x_s-\pi(x_s),\frac{\pi(x_s)-\pi(x_s+hd)}{h}} + \dotprod{d,\pi(x_s)-\pi(x_s+hd)} \le 0.
\end{align*}
By continuity of $\pi(\cdot)$, the right term goes to zero as $h$ goes to zero. So, taking the limit with $h$ going to zero yields
\begin{align}\label{eq:orth_prop_1}
 \dotprod{x_s-\pi(x_s),\lim_{h \to 0} \frac{\pi(x_s)-\pi(x_s+hd)}{h}} = -\dotprod{x_s- \pi(x_s), \partial\pi(x_s,\dot x_s)} \le 0.
\end{align}
Applying again Lemma~\ref{lem:orthogon_proj}-1. with $P = \pi$ to $x = x_s$, $d = \dot x_s$ and $y = \pi(x_s+h\dot x_s)$ gives
\begin{align*}
 \dotprod{x_s-\pi(x_s+h\dot x_s),\pi(x_s+h\dot x_s)-\pi(x_s)} \le 0.
\end{align*}
Dividing by $h$ which we take to $0$, it becomes
\begin{align}\label{eq:orth_prop_2}
 \dotprod{x_s-\pi(x_s),\partial \pi(x_s,\dot x_s)} \le 0.
\end{align}
Combining \eqref{eq:orth_prop_1} and \eqref{eq:orth_prop_2}, we get
\begin{equation}\label{eq:non_diff_2}
    \dotprod{x_s-\pi(x_s),\partial \pi(x_s,\dot x_s)} =0.
\end{equation}
Applying Lemma~\ref{lem:orthogon_proj}-2. to $x = x_s$ and $d = h\dot x_s$, $h > 0$, we get 
\begin{equation*}
        \dotprod{\pi( x_s+h\dot x_s)-\pi( x_s),h\dot x_s} \ge 0.
\end{equation*}
Dividing each side by $h^2$, we obtain
\begin{equation*}
        \dotprod{ \frac{\pi( x_s+h\dot x_s)-\pi( x_s)}{h},\dot x_s} \ge 0.
\end{equation*}
Taking $h \to 0$, we deduce
\begin{equation}\label{eq:non_diff_3}
    \dotprod{\partial \pi( x_s,\dot x_s),\dot x_s} \ge 0.
\end{equation}
Plugging \eqref{eq:non_diff_2} and \eqref{eq:non_diff_3} in \eqref{eq:non_diff_1}, we deduce 
\begin{eqnarray}
    \begin{aligned}\label{eq:sqc_proj_1_pre}
    \dot{\phi}(s) &=\frac{da_s}{ds}(f(x_s)-f^\ast) + \frac{1}{2}\frac{db_s}{ds}\norm{\pi(x_s) - z_s}^2\\
    &+ \eta a_s\dotprod{\nabla f(x_s),z_s-x_s} +\eta' b_s\dotprod{ z_s-\pi(x_s),x_s-z_s}.
    \end{aligned}
\end{eqnarray}

Also, using Cauchy-Swcharz we have
\begin{align*}
    \dotprod{ z_s-\pi(x_s),x_s - z_s} &= \dotprod{ z_s-\pi(x_s),x_s -\pi(x_s)+\pi(x_s)- z_s}\\
    &= \dotprod{ z_s-\pi(x_s),x_s -\pi(x_s)} - \norm{\pi(x_s)-z_s}^2\\
    &\le \frac{1}{2}\norm{x_s - \pi(x_s)}^2 - \frac{1}{2}\norm{\pi(x_s)-z_s}^2.
\end{align*}
Injected in \eqref{eq:sqc_proj_1_pre}, it becomes

\begin{align}
        \dot{\phi}(s) &\le \frac{da_s}{ds}(f(x_s)-f^\ast) + \frac{1}{2}\frac{db_s}{ds}\norm{\pi(x_s) - z_s}^2\nonumber \\
    &+a_s \eta \dotprod{ \nabla f(x_s) ,z_s-x_s} + \frac{ b_s \eta' }{2}\norm{x_s - \pi(x_s)}^2 - \frac{ b_s \eta' }{2}\norm{\pi(x_s)-z_s}^2.\label{eq:sqc_proj_1}
\end{align}
We now address the term $\mathbb{E}_{\xi} \varphi(\overline{x}_s+ G(\overline{x}_s)) - \varphi(\overline{x}_s)$.

\begin{equation}
    \begin{aligned}\label{eq:sqc_proj_2}
       &\mathbb{E}_{\xi} \varphi(\overline{x}_s+ G(\overline{x}_s)) - \varphi(\overline{x}_s)\\
   &= \mathbb{E}_{\xi}[a_s(f(x_s - \gamma \nabla f(x_s,\xi)) - f(x_s))] \\
   &+ \frac{b_s}{2}\mathbb{E}_{\xi}[\norm{ z_s - \gamma'\nabla f(x_s,\xi) - \pi(x_s - \gamma\nabla f(x_s,\xi))}^2 - \norm{z_s -\pi(x_s)}^2]\\
    &= a_s\mathbb{E}_{\xi}[f(x_s - \gamma \nabla f(x_s,\xi)) - f(x_s)] + \frac{b_s \gamma^{'2}}{2}\mathbb{E}_{\xi}[\norm{\nabla f(x_s,\xi)}^2]  \\
    & + \frac{b_s}{2}\mathbb{E}_{\xi}\left[\norm{z_s -  \pi(x_s - \gamma\nabla f(x_s,\xi))}^2 - \norm{z_s - \pi(x_s)}^2 \right]\\
    &+ b_s\gamma' \E{\dotprod{\nabla f(x_s,\xi), \pi(x_s - \gamma\nabla f(x_s,\xi)) - z_s}} \\
     &\le a_s\mathbb{E}_{\xi}[f(x_s - \gamma \nabla f(x_s,\xi)) - f(x_s)] + \frac{b_s \gamma^{'2}}{2}\mathbb{E}_{\xi}[\norm{\nabla f(x_s,\xi)}^2] + b_s\gamma' \dotprod{\nabla f(x_s), \pi(x_s) - x_s } \\
     &  + b_s\gamma' \dotprod{\nabla f(x_s), x_s - z_s}+ b_s\gamma' \E{\dotprod{\nabla f(x_s,\xi), \pi(x_s - \gamma\nabla f(x_s,\xi)) - \pi(x_s)}}\\
     &+ \frac{b_s}{2}\E{\norm{\pi(x_s) - \pi(x_s - \gamma \nabla f(x_s,\xi))}^2}+ b_s \E{\dotprod{z_s - \pi(x_s),\pi(x_s) - \pi(x_s - \gamma \nabla f(x_s,\xi)}}\\
     & \le  a_s \gamma( \gamma \rho L-2)\frac{1}{2}\norm{\nabla f(x_s)}^2+ \frac{b_s \gamma^{'2}}{2}\mathbb{E}_{\xi}[\norm{\nabla f(x_s,\xi)}^2]   + b_s\gamma' \dotprod{\nabla f(x_s), x_s - z_s} \\
     &- b_s\gamma'\bpar{f(x_s)-f^\ast + \frac{\mu}{2}\norm{x_s-\pi(x_s)}^2}+ b_s\gamma' \E{\dotprod{\nabla f(x_s,\xi), \pi(x_s - \gamma\nabla f(x_s,\xi)) - \pi(x_s)}}\\
     &+ \frac{b_s}{2}\E{\norm{\pi(x_s) - \pi(x_s - \gamma \nabla f(x_s,\xi))}^2}+ b_s \E{\dotprod{z_s - \pi(x_s),\pi(x_s) - \pi(x_s - \gamma \nabla f(x_s,\xi)}}.
    \end{aligned}
\end{equation}

where we used $\mathbb{E}_{\xi}[\nabla f(x_s,\xi)] = \nabla f(x_s)$, Lemma~\ref{lem:descent_sgc} and Assumption~\hyperref[ass:sgc]{(SQC'$_{\mu}$)}.
As $\arg \min f$ is convex, the projection $\pi(\cdot)$ is $1$-Lipschitz, so
\begin{equation}
\norm{\pi(x_s) - \pi(x_s - \gamma \nabla f(x_s,\xi))} \le \gamma\norm{\nabla f(x_s,\xi)}.
\end{equation}
Using this $1$-Lipschitz property of $\pi(\cdot)$ and Cauchy Schwartz inequalities, we have
\begin{align}
     &b_s\gamma' \dotprod{\nabla f(x_s,\xi), \pi(x_s - \gamma\nabla f(x_s,\xi)) - \pi(x_s)} + \frac{b_s}{2}\norm{\pi(x_s) - \pi(x_s - \gamma \nabla f(x_s,\xi))}^2\nonumber \\
     &+ b_s \dotprod{z_s - \pi(x_s),\pi(x_s) - \pi(x_s - \gamma \nabla f(x_s,\xi)}\nonumber\\
     &\le b_s \gamma' \norm{\nabla f(x_s,\xi)}\norm{\pi(x_s) - \pi(x_s - \gamma \nabla f(x_s,\xi))}  + \frac{b_s}{2}\norm{\pi(x_s) - \pi(x_s - \gamma \nabla f(x_s,\xi))}^2\nonumber \\
     &+b_s\norm{z_ s-\pi(x_s)} \norm{\pi(x_s) - \pi(x_s - \gamma \nabla f(x_s,\xi))}\nonumber\\
     &\le b_s \gamma'  \gamma \norm{\nabla f(x_s,\xi)}^2 + \frac{b_s\gamma^2}{2}\norm{\nabla f(x_s,\xi)}^2 + \frac{b_s\lambda}{2}\norm{z_s - \pi(x_s)}^2 + \frac{b_s \gamma^2}{2\lambda}\norm{\nabla f(x_s,\xi)}^2 \nonumber\\
     &=b_s\gamma\bpar{\gamma' + \frac{\gamma}{2} + \frac{\gamma}{2\lambda}}\norm{\nabla f(x_s,\xi)}^2 + \frac{b_s \lambda}{2}\norm{z_s - \pi(x_s)}^2 \label{eq:sqc_proj_3}
     \end{align}
     for a parameter $\lambda > 0$, to be fixed later. Combining \eqref{eq:sqc_proj_1}-\eqref{eq:sqc_proj_2}-\eqref{eq:sqc_proj_3} and using Assumption~\hyperref[ass:sgc]{(SGC$_\rho$)}, we get 
     \begin{align}
     &\dot{\phi}(s)+  \mathbb{E}_{\xi} \varphi(\overline{x}_s+ G(\overline{x}_s)) - \varphi(\overline{x}_s)\nonumber\\
     &\le \bpar{\frac{da_s}{ds}-  b_s \gamma'}(f(x_s)-f^\ast \label{eq:sqc:proj_1}) \\
     &+\bpar{\frac{db_s}{ds} + b_s\lambda- b_s \eta'}\frac{1}{2}\norm{z_s-\pi(x_s)}^2  \label{eq:sqc:proj_2}\\
     &+ \bpar{
  b_s \eta'   -  \mu b_s \gamma' }\frac{1}{2}\norm{x_s-\pi(x_s)}^2\label{eq:sqc:proj_3}\\
     &+(a_s \eta - b_s \gamma')\dotprod{\nabla f(x_s),z_s - x_s}\label{eq:sqc:proj_4}\\
     &+( \rho b_s \gamma^{'2} +2\rho b_s\gamma\bpar{\gamma' + \frac{\gamma}{2} + \frac{\gamma}{2\lambda}} + a_s \gamma( \gamma \rho L-2))\frac{1}{2}\norm{\nabla f(x_s)}^2.\label{eq:sqc:proj_5}
\end{align}

We choose $\eta$ of the form
$$\eta = \frac{\alpha}{\rho}\sqrt{\frac{\mu}{L}}.$$
We then cancel \eqref{eq:sqc:proj_1}, \eqref{eq:sqc:proj_3} and \eqref{eq:sqc:proj_4}
$$\frac{da_s}{ds}- b_s \gamma' = 0,\quad   b_s \eta'   -\mu  b_s \gamma' = 0, \quad
     a_s \eta - b_s \gamma' = 0, $$
which is ensured by choosing
$$a_s = e^{\eta s} = e^{\frac{\alpha}{\rho}\sqrt{\frac{\mu}{L}}s},\quad \eta' = \mu \gamma',\quad b_s = \frac{\eta}{\gamma'}e^{\eta s}=  \frac{\eta}{\gamma'}e^{\frac{\alpha}{\rho}\sqrt{\frac{\mu}{L}}s}.$$
 We now want to cancel \eqref{eq:sqc:proj_2}
\begin{align}
    \frac{db_s}{ds} + b_s\lambda- b_s \eta' = 0 &\Leftrightarrow \eta + \lambda - \eta' = 0 \Leftrightarrow \lambda  =  \eta'-\eta.\nonumber
\end{align}
Fix now $\eta' = \frac{1}{\rho}\sqrt{\frac{\mu}{L}}$ and $\gamma = \frac{1}{\rho L}$, which implies
$$\gamma' = \frac{1}{\rho\sqrt{\mu L}},\quad C_B = \alpha \mu,\quad \lambda = \frac{1-\alpha}{\rho}\sqrt{\frac{\mu}{L}}.$$
It remains to cancel \eqref{eq:sqc:proj_5}, which we recall to be
\begin{equation}
    \begin{aligned}\label{eq:sqc:proj_11}
      &\rho b_s \gamma^{'2} +2\rho b_s\gamma\bpar{\gamma' + \frac{\gamma}{2} + \frac{\gamma}{2\lambda}} + a_s \gamma( \rho\gamma L-2) \\
      &= a_s\left(\rho C_B \gamma'^2 + 2\rho C_B\gamma\bpar{\gamma' + \frac{\gamma}{2} + \frac{\gamma}{2\lambda}} +\gamma( \rho\gamma L-2)\right) \\
    &= a_s\left(\frac{\alpha}{\rho L} +  \alpha\frac{\mu}{ L} \bpar{\frac{2}{\rho \sqrt{\mu L}} + \frac{1}{\rho L} + \frac{1}{(1-\alpha)L}\sqrt{\frac{L}{\mu}}} - \frac{1}{\rho L}\right)\\
    &= \frac{a_s}{\rho L}\left(\alpha -1 + {\rho}\frac{\alpha }{(1-\alpha)} \sqrt{\frac{\mu}{L}}+ 2\alpha\sqrt{\frac{\mu}{L}} +\alpha\frac{\mu}{L}\right).
    \end{aligned}
\end{equation}

The term $\frac{\alpha }{(1-\alpha)} \sqrt{\frac{\mu}{L}}$ requires a special care, as it grows with $\alpha$ getting closer to $1$. Trying to make $\alpha-1 +\frac{\alpha }{(1-\alpha)} \sqrt{\frac{\mu}{L}} $ non positive, we see that a good choice of $\alpha$ is of the form $1-c\bpar{\frac{\mu}{L}}^{\frac{1}{4}}$. Plugging this choice in \eqref{eq:sqc:proj_11}$\times \frac{\rho L}{a_s}$, we have
\begin{align}
    &-c \bpar{\frac{\mu}{L}}^{\frac{1}{4}} + {\rho}\bpar{1-c\bpar{\frac{\mu}{L}}^{\frac{1}{4}}}\frac{1}{c}\bpar{\frac{\mu}{L}}^{\frac{1}{4}} + 2\bpar{1-c\bpar{\frac{\mu}{L}}^{\frac{1}{4}}}\sqrt{\frac{\mu}{L}} + \bpar{1-c\bpar{\frac{\mu}{L}}^{\frac{1}{4}}}\frac{\mu}{L}\\
    &\le \bpar{\frac{\mu}{L}}^{\frac{1}{4}}  \bpar{-c + \frac{\rho}{c} -  {\rho}\bpar{\frac{\mu}{L}}^{\frac{1}{4}} + 2 \bpar{\frac{\mu}{L}}^{\frac{1}{4}}  -2c\sqrt{\frac{\mu}{L}} + \bpar{\frac{\mu}{L}}^{\frac{3}{4}}-c\frac{\mu}{L}}.
\end{align}
As long as $c \ge \frac{1}{2}$, one has $-2c\sqrt{\frac{\mu}{L}} + \bpar{\frac{\mu}{L}}^{\frac{3}{4}}-c\frac{\mu}{L} \le 0$ because $\frac{\mu}{L} \le 1$. So, it suffices to ensure
\begin{equation}\label{eq:sqc:proj_12}
    -c + \frac{\rho}{c}  + ( 2-  {\rho})\bpar{\frac{\mu}{L}}^{\frac{1}{4}} \le 0.  
\end{equation}
One has $ -c + \frac{\rho}{c}  + ( 2-  {\rho})\bpar{\frac{\mu}{L}}^{\frac{1}{4}} \le -c + \frac{\rho}{c} +1$, such that
the choice $c = \frac{1 + \sqrt{1+4\rho}}{2}$ suffices to verify \eqref{eq:sqc:proj_12}. So, we fix $ \alpha = 1-\bpar{ \frac{1 + \sqrt{1+4\rho}}{2}}\bpar{\frac{\mu}{L}}^{\frac{1}{4}}$, which is non positive as long as 
\begin{equation*}
    \frac{\mu}{L} \le \bpar{ \frac{1 + \sqrt{1+4\rho}}{2}}^{-4}
\end{equation*}
In the deterministic gradient case, \textit{i.e.} $\rho = 1$, this condition becomes $ \frac{\mu}{L} \le \bpar{ \frac{1 + \sqrt{1+4\rho}}{2}}^{-4} = \bpar{ \frac{1 + \sqrt{5}}{2}}^{-4} \approx 0.15$.
In conclusion, if $\gamma = \frac{1}{ \rho L}$, $\gamma' = \frac{1}{\rho \sqrt{\mu L}}$, $\eta = \frac{1-\frac{1 + \sqrt{1+4\rho}}{2}\bpar{\frac{\mu}{L}}^{1/4}}{\rho}\sqrt{\frac{\mu}{L}}$ $\eta' = \frac{1}{\rho}\sqrt{\frac{\mu}{L}}$ and $C_B = \bpar{1-\frac{1 + \sqrt{1+4\rho}}{2}\bpar{\frac{\mu}{L}}^{1/4}}\mu$ we have
$$\forall s \in [0,t] \backslash \AA,\quad \dot{\phi}(s) + \mathbb{E}_{\xi}[(\varphi(x_{s} + G(x_{s},\xi))-\varphi(x_{s}))] \le 0.$$
According to Proposition~\ref{prop:calc_sto_nondiff}, we deduce
$$ \phi(t) \le \phi(0) + M_t,$$
where $\mart$ is a martingale such that $\E{M_t} = 0$ for all $t\in \R_+$. This in turn implies
\begin{align}
    e^{\eta t}(f(x_t)-f^\ast) \le \phi(t) \le \phi(0) + M_t.
\end{align}
Evaluating the latter at $t = T_k$ and taking expectation, thanks to the stopping Theorem~\ref{thm:martingal_stopping} we have
\begin{align*}
   \E{e^{\eta T_k}(f(\tilde x_k)-f^\ast)} &= \E{e^{\bpar{1-\frac{1 + \sqrt{1+4\rho}}{2})\bpar{\frac{\mu}{L}}^{1/4}}\frac{1}{\rho}\sqrt{\frac{\mu}{L}}T_k}(f(\tilde x_k)-f^\ast)}\\
   &\le f(x_0)-f^\ast +\mu\bpar{1-\frac{1 + \sqrt{1+4\rho}}{2}\bpar{\frac{\mu}{L}}^{1/4}}\norm{x_0-\pi(x_0)}^2.
\end{align*}


\newpage

\appendix

\section{Details on Figure~\ref{fig:figure_sqc}}\label{app:fig_detail}

\paragraph{Functions Displayed on the Left Side of Figure~\ref{fig:figure_sqc}} The function displayed on the left side of Figure~\ref{fig:figure_sqc} is 
$$f(t) =  0.5\cdot(t +0.15\sin(5t))^2, $$
such that 
$$f'(t) = (t +0.15\sin(5t)) \cdot (1+ 0.75\cos(5t)).$$
We have $(1+ 0.75\cos(5t)) > 0$. Noting $\varphi(t) = (t +0.15\sin(5t))$, it is elementary to check $\varphi'(t) > 0$, and as $\varphi(0) = 0$, we deduce $f$ only cancels at zero, which is its minimum. Then, remark
$$ \frac{f'(t)^2}{2 f(t)} = (1+0.75\cos(5t))^2 > 0.25^2,$$
which implies $f$ verifies Assumption~\hyperref[ass:pl]{(PL$_\mu$)}. According to \citep[Theorem 6]{hermant2024study}, because it is a 1d problem with unique minimizer, $f$ also verifies Assumption~\hyperref[ass:sqc]{(SQC$_{\tau,\mu}$)}.
\paragraph{Functions Displayed on the Right Side of Figure~\ref{fig:figure_sqc}}
The function displayed on the right side of Figure~\ref{fig:figure_sqc} is $h:\R^2 \to \R$ such that
\begin{equation*}
    h(x) = f(\norm{x})g\bpar{\frac{x}{\norm{x}}},
\end{equation*}
with $f(t) = 0.5t^2$, and $g(x_1,x_2) = \frac{1}{N}\sum_{i=1}^N \left[ a_i\sin(b_i x_1) + c_i\cos(d_ix_2) \right] +1$, with $N= 10$, $a_i,c_i$ are samples uniformly from a uniform law on $[-15,15]$, and $b_i,d_i$ are sampled uniformly from a uniform law on $[-25,25]$. The function $f$ verifies Assumption~\hyperref[ass:sqc]{(SQC$_{\tau,\mu}$)}, see \citep[Section A.2]{hermant2024study}.
\section{Proof of Proposition~\ref{prop:discretization_constant_param}. }\label{app:discretization}
The proof generalizes \citep{even2021continuized}, which addressed specific choices of parameters.
\paragraph{Step 1}
If we have
\begin{align}\left\{
    \begin{array}{ll}
        dx_t &= \eta(z_t-x_t)dt \\
        dz_t &= \eta'(x_t - z_t)dt
    \end{array}
\right.
\end{align}
integrating between $t_0$ and $t \ge t_0$, assuming $\eta' > -\eta$, we have
\begin{align*}
    x(t) &= \frac{\frac{\eta'}{\eta}x(t_0)+z(t_0)}{\frac{\eta + \eta'}{\eta}} + \frac{x(t_0)-z(t_0)}{\frac{\eta + \eta'}{\eta}} e^{-(\eta + \eta')(t-t_0)}\\
    &= x(t_0) + \frac{\eta}{\eta +\eta'}\bpar{1 -e^{-(\eta + \eta')(t-t_0)} }(z(t_0) - x(t_0))
\end{align*}
\begin{align*}
    z(t) &= \frac{x(t_0)+\frac{\eta}{\eta'}z(t_0)}{\frac{\eta + \eta'}{\eta'}} + \frac{z(t_0)-x(t_0)}{\frac{\eta + \eta'}{\eta'}} e^{-(\eta + \eta')(t-t_0)}\\
    &= z(t_0) + \frac{\eta'}{\eta + \eta'}\bpar{1 -e^{-(\eta + \eta')(t-t_0)} }(x(t_0) - z(t_0))
\end{align*}
\paragraph{Step 2}
Noting $t_0 = T_k$ and $t = T_{k+1^{-}}$, we have
\begin{align*}
    \tilde{y}_k := x(T_{k+1^{-}}) &= (z(T_k)-x(T_k))\frac{\eta}{\eta + \eta'}\bpar{1 -e^{-(\eta + \eta')(T_{k+1} - T_k)} } + x(T_k) \\
    &:= (\tilde{z}_{k} - \tilde{x}_k)\frac{\eta}{\eta+\eta'}\bpar{1 -e^{-(\eta + \eta')(T_{k+1} - T_k)} } + \tilde{x}_k
\end{align*}
Now, because of the gradient step at time $T_{k+1}$,
\begin{align*}
    \tilde{x}_{k+1} &:= x(T_{k+1})\\ &= (z(T_k)-x(T_k))\frac{\eta}{\eta+ \eta'}\bpar{1 -e^{-(\eta + \eta')(T_{k+1} - T_k)} } + x(T_k) - \gamma \nabla f(x(T_{k+1^{-}},\xi_{k+1}) \\
    &= \tilde{y}_k  - \gamma \nabla f(\tilde{y}_k,\xi_{k+1})
\end{align*}
Also, as $ \tilde{y}_k = (\tilde{z}_{k} - \tilde{x}_k)\alpha + \tilde{x}_k \Rightarrow \tilde{x}_k = \frac{\tilde{y}_k - \tilde{z}_{k} \alpha}{1-\alpha} = \tilde{z}_{k} + \frac{\tilde{y}_k-\tilde{z}_{k}}{1-\alpha}$, in our case $1-\alpha = 1- \frac{\eta}{\eta+ \eta'}\bpar{1 -e^{-(\eta + \eta')(T_{k+1} - T_k)} } = \frac{1}{\eta + \eta'}\bpar{\eta e^{-(\eta+\eta')(T_{k+1} - T_k)}+ \eta'}$, we deduce
\begin{align*}
    z(T_{k+1^{-}})&= z(T_k) + \frac{\eta'}{\eta + \eta'}\bpar{1 -e^{-(\eta + \eta')(T_{k+1}-T_k)} }(x(T_k) - z(T_k))\\
    &= z(T_k) + \frac{\eta'}{\eta+\eta'}\bpar{1 -e^{-(\eta + \eta')(T_{k+1}-T_k)} }\frac{(x(T_{k+1}^{-}) - z(T_k))}{\frac{1}{\eta + \eta'}\bpar{\eta e^{-(\eta+\eta')(T_{k+1}-T_k)}+ \eta'}}\\
    &= z(T_k) + \eta' \frac{\bpar{1- e^{-(\eta+\eta')(T_{k+1}-T_k)}}}{\eta' + \eta e^{-(\eta+\eta')(T_{k+1}-T_k)}}(x(T_{k+1}^{-}) - z(T_k)).
\end{align*}
Finally,
\begin{align*}
    \tilde{z}_{k+1} &:= z(T_{k+1}) = z(T_k) + \eta' \frac{\bpar{1- e^{-(\eta+\eta')(T_{k+1}-T_k)}}}{\eta' +\eta e^{-(\eta+\eta')(T_{k+1}-T_k)}}(x(T_{k+1}^{-}) - z(T_k)) - \gamma' \nabla f(x(T_{k+1}^{-}),\xi_{k+1})\\
    &= \tilde{z}_{k} + \eta' \frac{\bpar{1- e^{-(\eta+\eta')(T_{k+1}-T_k)}}}{\eta' + \eta e^{-(\eta+\eta')(T_{k+1}-T_k)}}(\tilde{y}_k - \tilde{z}_{k}) - \gamma' \nabla f(\tilde{y}_k,\xi_{k+1}).
\end{align*}
Finally, the algorithm is
\begin{align*}\left\{
    \begin{array}{ll}
        \tilde{y}_k &= (\tilde{z}_{k} - \tilde{x}_k)\frac{\eta}{\eta+\eta'}\bpar{1 -e^{-(\eta + \eta')(T_{k+1}-T_k)} } + \tilde{x}_k \\
        \tilde{x}_{k+1} &=  \tilde{y}_k  - \gamma \nabla f(\tilde{y}_k,\xi_{k+1})\\
        \tilde{z}_{k+1} &=   \tilde{z}_{k} + \eta' \frac{\bpar{1- e^{-(\eta+\eta')(T_{k+1}-T_k)}}}{\eta' + \eta e^{-(\eta+\eta')(T_{k+1}-T_k)}}(\tilde{y}_k - \tilde{z}_{k}) - \gamma' \nabla f(\tilde{y}_k,\xi_{k+1})
    \end{array}
\right.
\end{align*}

\section{Convergence Results Trajectory-Wise}
\subsection{Proof of Proposition~\ref{prop:conv_high_proba}}\label{app:sqc_high_proba}
Before proving Proposition~\ref{prop:conv_high_proba}, we prove the Chernov inequality to Gamma laws (Lemma~\ref{prop:chernov}. This is rather a classical result, but we provide a proof for the sake of completeness.

\noindent
\textbf{Proof of Lemma~\ref{prop:chernov} \:}
The Chernov's inequality is deduced by applying the Markov inequality to the right transformation of the random variable. Indeed for a non negative random variable $X$, for $t< 0$ and $a > 0$ we have
\begin{align}\label{eq:chernov}
    \P(X \le a) = \P(e^{t X} \ge e^{ta}) \le \E{e^{t X}}e^{-ta} \le \inf_{t<0} \E{e^{tX}}e^{-ta}.
\end{align}
The moment generative function of a Gamma law with parameters $(n,1)$ is well known, namely $\E{e^{t T_n}} = (1-t)^{-n}$, for $t < 1$. Applying \eqref{eq:chernov} to $X = T_n$, $a = c\E{T_n} = cn$ with $c\le 1$, we deduce
\begin{align*}
   \P(T_n \le cn) &\le  \inf_t (1-t)^{-n}e^{-cnt}\\
   &= \inf_t e^{-n(ct + \log(1-t))}.
\end{align*}
The infimum is reach at $t = 1-\frac{1}{c}$, such that
\begin{align*}
     \P(T_n \le cn) \le e^{-n(c-1-\log(c))}.
\end{align*}

\noindent
\textbf{Proof of Proposition~\ref{prop:conv_high_proba} \:}
      By the Markov Property and because $\mathbb{E}\left[e^{\beta T_k}Y\right] \le K_0$, for some $C_0 > 1$ we have
    \begin{equation*}
        \P\bpar{e^{\beta T_k}Y > C_0} \le  \frac{  \mathbb{E}\left[e^{\beta T_k}Y\right]}{C_0} \le \frac{K_0}{\varepsilon C_0}.
    \end{equation*}
    Set $C_0 := c_0 \frac{1}{\varepsilon}K_0$, $c_0 \ge 1$. With probability $1- \frac{1}{c_0}$, we have
    \begin{equation}\label{eq:highprob_1}
        e^{\beta T_k}Y  \le \frac{c_0}{\varepsilon}K_0.
    \end{equation}
    Using that $T_k \sim \Gamma(k,1)$, we apply Lemma~\ref{prop:chernov} (Chernov inequality) with  $c = c_1 \in (0,1)$, such that
    \begin{equation*}
          \P(T_k \le c_1k)\le e^{-(c_1-1-\log(c_1))k}.
    \end{equation*}
    We thus have 
    \begin{equation}\label{eq:highprob_2}
        T_k \ge c_1k
    \end{equation} with probability at least $1-e^{-(c_1-1-\log(c_1))k}$.
    Combining \eqref{eq:highprob_1} and \eqref{eq:highprob_2}, and using that for two events $A$, $B$ we have $\P(A\cap B) \ge \P(A) + \P(B) - 1$, we deduce that with probability $1-\frac{1}{c_0} - e^{-(c_1-1-\log(c_1))k}$
    \begin{equation*}
        Y \le \frac{c_0}{\varepsilon}K_0e^{-\beta c_1k}.
    \end{equation*}

\subsection{Proof of Proposition~\ref{prop:almost_sure}}\label{app:sqc_almostsure}

   Let $\tilde \varepsilon' > 0$. By assumption, we have
   \begin{equation*}       \forall k \in \N^\ast,~ \mathbb{E}\left[e^{\beta T_k}Y_k\right] \le K_0,
         \end{equation*}
         which combined with the Markov inequality yields
         \begin{equation}
             \P\bpar{e^{\beta T_k} Y_k\ge K_0e^{\tilde \varepsilon'  k}} \le \frac{\E{e^{\beta T_k} Y_k}}{K_0e^{\tilde \varepsilon'  k}} \le e^{-\tilde \varepsilon'  k}.
         \end{equation}
         Applying the Chernov inequality (Lemma~\ref{prop:chernov}) with some $c \in (0,1)$, we have
         \begin{equation}
             \P(e^{\beta T_k} \le e^{\beta ck}) = \P (T_k \le c \beta) \le e^{-n \Delta_c},
         \end{equation}
         where $\Delta_c := c-1-\log(c)$.  So,
         \begin{align}\label{eq:a.s.:3}
             \P\bpar{ e^{\beta T_k} \le e^{\beta ck}~  \bigcup e^{\beta T_k} \ge K_0e^{\tilde \varepsilon'  k}} \le \P\bpar{ e^{\beta T_k} \le e^{\beta ck}} + \P \bpar{e^{\beta T_k} \ge K_0e^{\tilde \varepsilon'  k}} \le e^{-n \Delta_c} +   e^{-\tilde \varepsilon'  k}.
         \end{align}
         Using \eqref{eq:a.s.:3}, we observe that
         \begin{align} \label{eq:a.s.:4}
             \sum_{k\ge 0} \P\bpar{ e^{\beta T_k} \le e^{\beta ck}~  \bigcup e^{\beta T_k} \ge K_0e^{\tilde \varepsilon'  k}} \le \sum_{k\ge 0} \bpar{e^{-n \Delta_c} +   e^{-\tilde \varepsilon'  k}} < + \infty.
         \end{align}
         By Borell-Cantelli's Lemma, \eqref{eq:a.s.:4} induces $\P\bpar{\lim \sup_k \left\{ e^{\beta T_k} \le e^{\beta ck}~  \bigcup e^{\beta T_k}E_n \ge K_0e^{\tilde \varepsilon'  k} \right\}} = 0 $, which implies
         \begin{equation}
             \begin{aligned}
                   & \P \bpar{\lim \inf_k \left\{ e^{\beta T_k}\le e^{\beta ck} \right\}^C~  \bigcap \left\{ e^{\beta T_k}Y_k \ge K_0e^{\tilde \varepsilon'  k} \right\}^C} \\&= \P \bpar{\lim \inf _k\left\{ e^{-\beta T_k} < e^{-\beta ck} \right\}~  \bigcap \left\{ e^{\beta T_k}Y_k< K_0e^{\tilde \varepsilon'  k} \right\})}\\
         &= 1     
             \end{aligned}
         \end{equation}
In particular, it implies that almost surely, we have
\begin{equation}
  \P \bpar{ \lim \inf _k\bpar{ Y_k< K_0e ^{-(\beta c - \tilde \varepsilon' )k} }}=1.
\end{equation}
Then, for almost every $\omega \in \Omega$, there exists $K(\omega) \in \N$, such that for all $k\ge K(\omega)$, we have
\begin{equation}
    Y_k< K_0e ^{-(\beta c - \tilde \varepsilon' )k},
\end{equation}
 from which we have
 \begin{equation}
    Y_ke^{(\beta c - 2\tilde \varepsilon' )k} \le K_0 e^{-\tilde \varepsilon'  k} \to_{k\to +\infty} 0.
 \end{equation}
 We fix now $2 \tilde \varepsilon' =\frac{1}{2} \beta c$, such that $\beta c - 2\tilde \varepsilon' = \frac{\beta c}{2} $. Fixing $c' = \frac{c}{2}$, we get that for any $c' \in (0,1)$, almost surely
          \begin{equation*}
    Y_k = o\bpar{e^{-\beta c' k}}.
\end{equation*}

\section{Technical Lemmas}

\subsection{Proof of Lemma~\ref{lem:descent_sgc}}\label{app:descent_lemma}
    By Assumption~\hyperref[ass:l_smooth]{(L-smooth)}
    \begin{eqnarray}
        \begin{aligned}\label{eq:descent_lemma_0}
                   f(x-\gamma \nabla f(x,\xi)) &\le f(x) + \dotprod{\nabla f(x),x-\gamma \nabla f(x,\xi) - x} + \frac{L}{2}\norm{ x-\gamma \nabla f(x,\xi) - x}^2\\
        &= f(x) -\gamma \dotprod{\nabla f(x),\nabla f(x,\xi)} + \frac{\gamma^2L}{2}\norm{\nabla f(x,\xi)}^2.
        \end{aligned}
    \end{eqnarray}
We take expectation on both sides of \eqref{eq:descent_lemma_0}
\begin{equation}\label{eq:descent_lemma_1}
    \E{f(x-\gamma \nabla f(x,\xi)) - f(x)} \le -\gamma \dotprod{\nabla f(x),\E{\nabla f(x,\xi)}} + \frac{\gamma^2L}{2}\E{\norm{\nabla f(x,\xi)}^2}.
\end{equation}
By Assumption~\hyperref[ass:sgc]{(SGC$_\rho$)}, we have
\begin{equation}\label{eq:descent_lemma_2}
    \dotprod{\nabla f(x),\E{\nabla f(x,\xi)}}= \norm{\nabla f(x)}^2,\quad \E{\norm{\nabla f(x,\xi)}^2} \le \rho \norm{\nabla f(x)}^2.
\end{equation}
Injecting \eqref{eq:descent_lemma_2} in \eqref{eq:descent_lemma_1}, we have
$$\E{f(x-\gamma \nabla f(x,\xi)) - f(x)} \le \gamma\left(\frac{L \rho}{2}\gamma - 1 \right) \lVert  \nabla f(x) \rVert^2.$$

\subsection{Proof of Lemma~\ref{lem:orthogon_proj}}\label{app:orthogon_proj}
    \textbf{Statement 1. \quad}By definition of the orthogonal projection, for $x\in U$ and $y \in K$
    $$ \norm{x-P(x)}^2 = \norm{x-y +y-P(x)}^2 \le \norm{x-y}^2.$$
    Developing the squared norm, we obtain
    $$ \norm{y-P(x)}^2 + 2\dotprod{x-y,y-P(x)} \le 0,$$
    which gives the result.

    \textbf{Statement 2. \quad}
        For $x,d \in U$, we have
    \begin{align*}
        &\dotprod{P(x+d)-P(x),d} \\
        &= \dotprod{P(x+d)- x,d} - \dotprod{P(x)-x,d} \\
        &= \frac{1}{2}\left[ \norm{P(x+d)-x}^2  + \norm{d}^2 - \norm{P(x+d)-(x+d)}^2 \right.\\
        &\left.- \norm{P(x)-x}^2 - \norm{d}^2 + \norm{P(x)-(x+d)}\right]\\
        &=\frac{1}{2}\left[ \norm{P(x+d)-x}^2  - \norm{P(x+d)-(x+d)}^2 - \norm{P(x)-x}^2+ \norm{P(x)-(x+d)}^2\right].
    \end{align*}
    Note that
    \begin{equation*}
 \norm{P(x+d)-x}^2  \ge  \norm{P(x)-x}^2, \quad \norm{P(x)-(x+d)}^2 \ge \norm{P(x+d)-(x+d)}^2,
    \end{equation*}
    such that  
    \begin{equation*}
        \dotprod{P(x+d)-P(x),d} \ge 0.
    \end{equation*}
\section{Background}
In this section, we give additional details about different formulations of momentum ODEs and algorithms.

\subsection{The Different Forms of Momentum}\label{app:momentum_form}
We give a very short description of different well known formulations of momentum algorithms.
 First, we note the seminal Polyak's Heavy Ball \citep{POLYAK19641}.
$$ \tilde x_{k+1} = \tilde x_k + \underbrace{\alpha_k(\tilde x_k-\tilde x_{k-1})}_{\text{momentum}} - \underbrace{\eta \nabla f(\tilde x_k)}_{\text{gradient step}}. $$
This algorithm has the advantage of a rather simple and intuitive formulation. However, its performance can degrade in certain settings compared to Nesterov’s version \citep{ghadimi2015global,goujaud2025provable}. Two important formulations of the Nesterov momentum can be described as follows:
\[
\begin{array}{c}
\textbf{NM (ML version)} \\[0.2cm]
\left\{
\begin{aligned}
\tilde b_k &= p\,\tilde b_{k-1} + (1-\tau)\, \nabla f(\tilde x_{k-1})\\[0.15cm]
\tilde x_k &=\tilde x_{k-1} - s\bigl(\nabla f(\tilde x_{k-1}) + p \tilde b_k\bigr)
\end{aligned}
\right.
\end{array}
\quad\quad 
\begin{array}{c}
\textbf{NM (FISTA-style version)} \\[0.2cm]
\left\{
\begin{aligned}
&\tilde y_k = \tilde x_k + a_k(\tilde x_k -\tilde x_{k-1}) + b_k(\tilde x_k -\tilde y_{\,k-1}) \\[0.15cm]
&\tilde x_{k+1} =\tilde y_k - s\, \nabla f(\tilde y_k)
\end{aligned}
\right.
\end{array}
\]
The ML version coincides with the implementation commonly used in deep-learning frameworks, see \textit{e.g.} \citep{sutskever2013importance}. The other can be regarded as a standard formulation in a certain strand of the classical optimization literature, which we may loosely refer to as the ‘FISTA-style’ version \citep{beck2009fast}. \eqref{alg:nest_classic} generalizes these two version, see \citep[Appendix B.2]{hermantgradient}. Finally, note that the fundamental difference between Nesterov Momentum and Polyak's Heavy Ball can be understood through their distinct high-resolution ODE limits \citep{shi2022understanding}.

\subsection{Link Between Momentum ODEs}\label{app:ode_link}

The Heavy Ball equation is defined by the following equation
\begin{equation}\label{eq:HB}\tag{HB}
    \ddot{X}_t + \alpha(t) \dot{X}_t + \nabla f(X_t) = 0,
\end{equation}
where $\alpha(t) > 0$. Intuitively, it models a rolling object on the surface defined by the function, submitted to a friction which is parameterized by $\alpha(t)$. Another popular equation is the following \citep{alvarez2002second}
\begin{equation}\label{eq:HB_hess}\tag{HB-Hess}
    \ddot{X}_t + \alpha(t)\dot{X}_t + b\nabla^2 f(X_t) + \nabla f(X_t) = 0.
\end{equation}
Importantly, \eqref{eq:HB_hess} describes more accurately the behavior of the Nesterov accelerated gradient algorithm than \eqref{eq:HB}, see \citep{shi2022understanding}. In section~\ref{sec:further}, we introduced the following equation
\begin{align}\label{eq:nest_2_var_appendix}\left\{
    \begin{array}{ll}
        \dot{x}_t &= \eta_t(z_t-x_t) - \gamma_t \nabla f(x_t) \\
        \dot{z}_t &= \eta'_t(x_t-z_t) - \gamma'_t \nabla f(x_t) 
    \end{array}
\right.
\end{align}
with $\eta_t > 0$ for all $t \in \R_+$.
This formulation is more general than  \eqref{eq:HB_hess}, as we show next.
\begin{proposition}
    Let $\xz$ verify \eqref{eq:nest_2_var_appendix}. Then, it verifies
    \begin{equation}
    \ddot{x}_t + \bpar{\eta_t - \frac{\dot{\eta}_t}{\eta_t} + \eta_t' }\dot{x}_t + \gamma_t \nabla^2 f(x_t)\dot{x}_t + \bpar{ \eta'_t \gamma_t + \eta_t \gamma'_t -\frac{\dot{\eta}_t}{\eta_t} }\nabla f(x_t) = 0.
\end{equation}
\end{proposition}
\begin{proof}
    We have
\begin{align}
    &z_t = x_t + \frac{1}{\eta_t}\bpar{\dot{x}_t +  \gamma_t \nabla f(x_t)}\\
    &\dot{z}_t = \dot{x}_t - \frac{\dot{\eta}_t}{\eta_t^2}\bpar{\dot{x}_t +  \gamma_t \nabla f(x_t)} + \frac{1}{\eta_t}\bpar{\ddot{x}_t +  \gamma_t \nabla^2 f(x_t)\dot{x}_t}
\end{align}
From this and the second line of \eqref{eq:nest_2_var_appendix}
\begin{align}
     \dot{x}_t - \frac{\dot{\eta}_t}{\eta_t^2}\bpar{\dot{x}_t +  \gamma_t \nabla f(x_t)} + \frac{1}{\eta_t}\bpar{\ddot{x}_t +  \gamma_t \nabla^2 f(x_t) \dot{x}_t} &=  \eta'_t(x_t-z_t) - \gamma'_t \nabla f(x_t) \\
     &= -\frac{\eta'_t}{\eta_t}\dot{x}_t -  \frac{\eta'_t \gamma_t}{\eta_t}\nabla f(x_t) - \gamma'_t \nabla f(x_t).
\end{align}
We multiply by $\eta_t$ and rearrange, to get
\begin{equation}
    \ddot{x}_t + \bpar{\eta_t - \frac{\dot{\eta}_t}{\eta_t} + \eta_t' }\dot{x}_t + \gamma_t \nabla^2 f(x_t)\dot{x}_t + \bpar{ \eta'_t \gamma_t + \eta_t \gamma'_t -\frac{\dot{\eta}_t}{\eta_t} }\nabla f(x_t) = 0.
\end{equation}
\end{proof}

\section*{Acknowledgement}
This work was supported by PEPR PDE-AI. We thank Erell Gachon for her feedback on the introduction.
\vskip 0.2in
\bibliography{sample}

\end{document}